\documentclass{amsart}
\usepackage{amssymb,amsfonts}
\makeatletter
\newcommand\thankssymb[1]{\textsuperscript{\@fnsymbol{#1}}}
\makeatother
\usepackage[percent]{overpic}
\usepackage{thmtools} 
\usepackage{thm-restate}
\usepackage{lipsum}
\usepackage{url}
\usepackage{subfig}
\usepackage{hyperref}
\usepackage{xcolor}
\declaretheorem[name=Theorem,numberwithin=section]{thm}
\usepackage[all,arc]{xy}
\usepackage{enumerate}
\usepackage{color,soul}
    \sethlcolor{white}
\usepackage{mathrsfs}
\usepackage{pinlabel}
\usepackage{mathptmx}
\usepackage{anyfontsize}
\usepackage{t1enc}
\usepackage{rotating}
\usepackage{xparse}
\usepackage{graphicx}
\usepackage{comment}

\newtheorem{cor}[thm]{Corollary}
\newtheorem{prop}[thm]{Proposition}
\newtheorem{lem}[thm]{Lemma}
\newtheorem{conj}[thm]{Conjecture}
\newtheorem{quest}[thm]{Question}
\usepackage{graphicx}
\theoremstyle{definition}
\newtheorem{defn}[thm]{Definition}

\newtheorem{con}[thm]{Construction}
\newtheorem{exmp}[thm]{Example}

\usepackage{mathtools}
\DeclarePairedDelimiter{\ceil}{\lceil}{\rceil}
\theoremstyle{remark}
\newtheorem{rem}[thm]{Remark}

\newcommand{\Z}{\mathbb{Z}}
\newcommand{\R}{\mathbb{R}}

\makeatletter
\let\c@equation\c@thm
\makeatother
\numberwithin{equation}{section}

\title{Bridge number and meridional rank of knotted surfaces}
\author{Jason Joseph}\thanks{The first author was supported by the Max Planck Institute for Mathematics and NSF grant DMS-1664567 during part of this research and is currently supported by NSF-RTG grant NSF DMS-1745670.
}

\author{Puttipong Pongtanapaisan}\thanks{Research conducted for this paper is supported by the Pacific Institute for the Mathematical Sciences (PIMS). The research and findings may not reflect those of the Institute.}

\begin{document}

\maketitle

\textbf{Abstract:} 
The Meridional Rank Conjecture asks whether the bridge number of a knot in $S^3$ is equal to the minimal number of meridians needed to generate the fundamental group of its complement. In this paper we investigate the analogous conjecture for knotted surfaces in $S^4$. Towards this end, we give a construction to produce classical knots with quotients sending meridians to elements of any finite order and which detect their meridional ranks. We establish the equality of bridge number and meridional rank for these knots and knotted spheres obtained from them by twist-spinning. On the other hand, we show that the meridional rank of knotted spheres is not additive under connected sum, so that either bridge number also collapses, or meridional rank is not equal to bridge number for knotted spheres. We also show a relationship between the bridge numbers of welded knots and ribbon tori using the Tube map, and give applications to bridge trisections of knotted surfaces.

\section{Introduction}

In the classical setting, the \textbf{bridge number} $\beta(K)$ is a fundamental measure of complexity for a knot $K$ in $S^3$. The bridge number provides a comprehensible exhaustion of all knots; indeed, 2-bridge knots are the simplest of knots in many ways, and their classification by Schubert was a triumph of early knot theory \cite{schubert}. Cappell and Shaneson's Meridional Rank Conjecture posits that $\beta(K)$ is equal to the \textbf{meridional rank} $\mu(K)$, the minimal number of meridians needed to generate $\pi_1(S^3\backslash K)$. Tools such as knot contact homology and Coxeter quotients have been used to verify that the conjecture holds for several families of knots (see \cite{baader2020coxeter}, and references therein) but no counterexamples have been discovered. In this paper, we study the bridge numbers and meridional ranks of knotted surfaces in $S^4$.

The bridge number of a knotted surface is completely analogous to the classical case: it is the minimal number of minima of the surface taken over all embeddings in $S^4$. However, unlike the classical case, not much is known about the bridge number of knotted surfaces. Scharlemann showed that a sphere in $S^4$ with 4 criticial points is standard \cite{scharlemann}, but it is conceivable that a nontrivial sphere could have a single minimum and three or more maxima. Such a sphere would have group $\Z$, so by work of Freedman it would be topologically unknotted \cite{freedman}. Hence it is not known if $\beta(K)=1$ implies that $K$ is the unknot.

\textit{Twist-spinning} is an operation introduced by Zeeman \cite{zeeman} which produces a knotted sphere $\tau^m K$ from a classical knot $K\subseteq S^3$ and an integer $m$, called the twist index. By construction, $\beta(\tau^m K)\leq \beta (K)$. Similarly, $\mu(\tau^m K)\leq \mu(K)$, because the group of a twist-spun knot is a quotient of the classical knot group. Note that $m=\pm 1$ always yields an unknotted sphere.

\begin{quest} \label{quest1}
Given $K\subseteq S^3$, does there exist $m\neq\pm 1$ such that $\mu(\tau^m K)<\mu(K)$, or $\beta(\tau^m K)<\beta(K)$?
\end{quest}

In Theorem \ref{thm:twistspunrank} we find conditions on $K$ and $m$ so that the equality of $\mu(K)$ and $\beta(K)$ ensure the equality of all four of these values. To do so, we need to find quotients of knot groups which are compatible with the quotient maps $\pi_1(S^3\setminus K) \rightarrow \pi_1(S^4\setminus \tau^m K)$. In the case that $m$ is even, Coxeter quotients are sufficient, and we make use of large families of examples for which the MRC is known, due to Baader, Blair, and Kjuchukova in \cite{baader2020coxeter}, and these authors and Misev in \cite{baader2020bridge}. We will refer to these examples as \textit{BBKM knots}.

For odd-twist spinning, we adapt the construction of Brunner \cite{brunner}, utilized in \cite{baader2020coxeter}, to find quotients of classical knot groups sending meridians to $p$-cycles of any finite order $p$. This may be of independent interest, as it
works in situations where Coxeter quotients, the Alexander module, and kei colorings all fail. When applicable, we will mention how these techniques allow us to compute the meridional rank for more general deform-spun knots. We summarize these results in the following theorem.

\begin{restatable}{thm}{twistspun}
\label{thm:twistspunrank}
Let $m,n\in \Z$ with $|m|\neq 1, n\geq 2$. There exist infinitely many classical knots $K\subseteq S^3$ such that $\mu(\tau^m K) = \beta(\tau^m K) = n$.
\end{restatable}

This answers, for these examples, a question of Meier and Zupan regarding the \textit{bridge trisection index} $b(S)$, which can be regarded as the analogue of trisection genus \cite{gay2016trisecting} in the world of knotted surfaces. Note that this is not a bridge number in the sense of counting minima, but is related to it: $3\beta(S)-\chi(S)\leq b(S)$. Meier and Zupan showed that $b(S)$ can be arbitrarily large, and achieves every positive integer congruent to 0 or 1 mod 3. We reprove a theorem of Sato and Tanaka \cite{sato2020bridge} that the case of $b\equiv 2$ mod 3 is also achieved, using meridional rank instead of quandle colorings. Other applications to bridge trisections are given in Section \ref{sec:bridgetrisections}. 

Twist-spun knots can be used to exhibit interesting behaviors. It is an open question whether the meridional rank is $(-1)$-additive under connected sum of classical knots, whereas Schubert proved that the bridge number is $(-1)$-additive for classical knots. On the other hand, both $\mu(S)$ and $\beta(S)$ fail to be $(-1)$-additive for connected sum of knotted surfaces. Although he was working in the context of abstract knot groups, Maeda proved in \cite{maeda1977composition} that there exist knotted surfaces $S_1$ and $S_2$ of genus one such that $\mu(S_1) = \mu(S_2)=2$ and $\mu(S_1\# S_2) =2$. For the bridge number, there is an example due to Viro in \cite{viro1973local} of a knotted sphere $F$ with $\beta(F)=2$, such that connected sum with a standard projective plane $\R P^2$ is again a standard projective plane, hence $\beta(F\# \R P^2) = \beta(\R P^2)=1$. The $({-1})$-additivity of bridge number appears to remain open in the case of orientable knotted surfaces. However, using examples first studied by Kanenobu \cite{kanenobu1996weak} we show that the meridional rank of a connected sum of spheres can achieve any value in between the theoretical limits, so that either the meridional rank conjecture fails for knotted spheres, or bridge number also fails to be $(-1)$-additive. 

 \begin{restatable}{thm}{nonadditivity}\label{thm:nonadditivity}
 Let $p_1,\dots,p_n,q\geq1$ such that $max\{ p_i\} \leq q \leq \sum p_i -(n-1)$. There exist 2-knots $K_1,\dots ,K_n$, with $\mu(K_i)=p_i$ for all $i$ and such that $\mu(K_1\#\cdots\# K_n)=q$.

 \end{restatable}

 
 \begin{restatable}{cor}{nonaddcorollary}
Either bridge number fails to be $(-1)$-additive on 2-knots, or there exist 2-knots $K$ with $\mu(K)<\beta(K)$.
 \end{restatable}




For a higher dimensional analogue of the Meridional Rank Conjecture for ribbon $T^2$-knots, we first represent them by virtual knot diagrams via Satoh's \textit{Tube map} correspondence \cite{satoh2000virtual,kauffman2021virtual}. The quantities $\beta(S), b(S),$ and $\mu(S)$ can then be conveniently calculated on these generalized knot diagrams, which look just like classical knot diagrams with a new type of crossing. We will show that if a ribbon $T^2$-knot $F$ can be represented as a diagram obtained from a BBKM knot by performing appropriate local operations, then the meridional rank equals the bridge number.

\begin{restatable}{thm}{virtualBBKM}\label{thm:virtualbbkm}
Consider a BBKM knot diagram $D$ of a knot $K$ with $\mu(K)=n.$ If $D'$ is the result of performing some of the operations below, then $\mu(Tube(D'))\geq n:$
\begin{enumerate}
\item Virtualizing up to $k-1$ classical crossings in a twist region made up of $k$ crossings that corresponds to either a single Coxeter generator or two Coxeter generators, where $k$ is even.
        \item Performing any number of flank-switch moves anywhere.
    \item For each flank move performed on a twist region in a rational tangle corresponding to two Coxeter generators, adding at least a classical crossing to balance out.
\end{enumerate}

 \end{restatable}
 \begin{figure}[ht!]
\labellist
\small\hair 2pt

\pinlabel \text{$a$} at 904 529
\pinlabel \text{$c$} at 236 29
\pinlabel \text{$b$} at 584 609
\pinlabel \text{$c$} at 284 690
\endlabellist
 \centering
    \includegraphics[width=.7\textwidth]{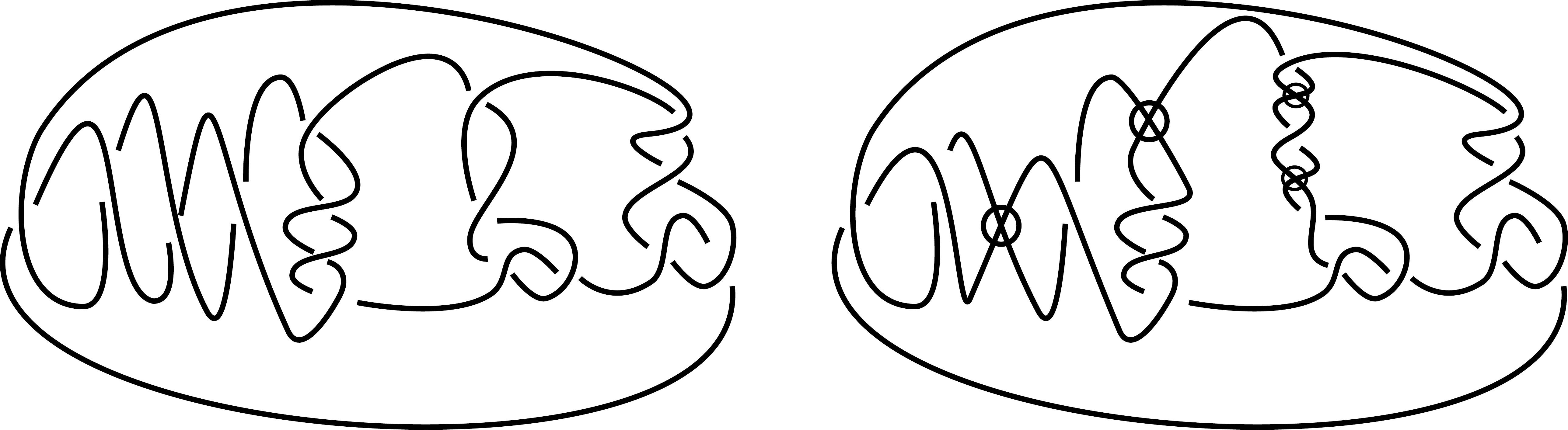}
    \caption{At left, a classical BBKM knot. At right, a virtual BBKM knot diagram, where the leftmost and the middle virtual crossings demonstrates situation (1), the rightmost two virtual crossings demonstrate situation (3) of Theorem \ref{thm:virtualbbkm}}
    \label{fig:modification}
\end{figure}
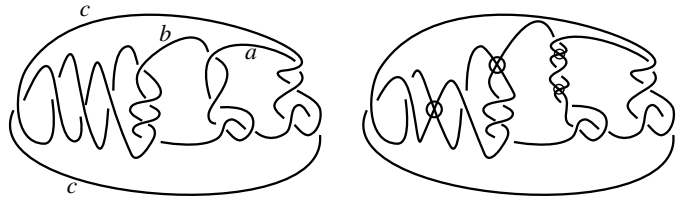

The \textbf{Wirtinger number} is a combinatorial upper bound for the bridge number of a classical knot, introduced by Blair et al. \cite{blair2020wirtinger}. In that work, they prove that the Wirtinger number is in fact equal to the bridge number. The second author's result on the Wirtinger number of virtual links \cite{pongtanapaisan2019wirtinger}, and Theorem \ref{thm:virtualbbkm} allows us to compute the exact values of the bridge numbers and meridional ranks of infinitely many ribbon $T^2$-knots.
 
 \begin{restatable}{cor}{VBBKM}
Let $K$ be a welded knot obtained from Theorem \ref{thm:virtualbbkm}, and let $F=Tube(K)$, a ribbon $T^2$-knot. Then, $\mu(F)=\beta(F)$.
 \end{restatable}
 
These modified BBKM knots with $n$ minima naturally give rise to tori in $S^4$ with $n$ minima, and in fact a bridge trisection diagrams with $3n$ arcs via Satoh's Tube map. This upper bound, together with Theorem \ref{thm:virtualbbkm}, shows that the welded bridge number can determine the bridge trisection number of knotted surfaces in certain cases.

 \begin{restatable}{cor}{bridgenumberribbon}
 \label{cor:bridgenumberribbon}
If $F$ is a ribbon $T^2$-knot satisfying the hypotheses of Theorem \ref{thm:virtualbbkm}, then $b(F)=3n$.
 \end{restatable}

 Due to the compatibility with amalgamation of Coxeter groups, the following corollary gives instances where we can guarantee that the meridional rank and the bridge number computed in this paper are $(-1)$-additive.

\begin{cor}
The meridional rank and the bridge number are $(-1)$-additive for even-twist spun BBKM knots and for ribbon surfaces obtained by applying Satoh's Tube map to virtual BBKM link diagrams.
\end{cor}

\subsection*{Organization}
This paper is organized as follows. In Section \ref{sec:prelims}, we define bridge number and meridional rank of classical knots and review some families of knots for which the MRC is known. In Section \ref{sec:pcycles} we develop a construction to build knots which can be labeled by $p$-cycles for any $p\geq 2$. We use this technique in Section \ref{sec:2-knots} to establish the MRC for a large family of 2-knots in Theorem \ref{thm:twistspunrank}. We then investigate the additivity of meridional rank under connected sum of 2-knots in Section \ref{sec:behaviorcsum}, and prove Theorem \ref{thm:nonadditivity}. In Section \ref{sec:weldedmrkconj}, we prove a relationship between the bridge numbers of certain welded knots and their associated knotted tori, and prove Theorem \ref{thm:virtualbbkm}. In Section \ref{sec:bridgetrisections}, we collect our applications to bridge trisections of knotted surfaces.

\subsection*{Acknowledgements}
The authors would like to thank Ryan Blair, Micah Chrisman, and Alexandra Kjuchukova, for helpful conversations and encouragement.

\section{Preliminaries}\label{sec:prelims}

\subsection{Meridional rank and bridge number of knots in $S^3$} Here we define the quantities in the classical Meridional Rank Conjecture.

\begin{defn}
Let $K:S^1\rightarrow S^3$ be a knot and let $N(K)$ be a tubular neighborhood of $K$. A fiber of $N(K)$ is a meridional disk $D=\{*\}\times D^2$ with $K\cap D=\{*\}$. A \textbf{meridian} of $K$ is an element of $\pi_1(S^3\setminus N(K))$ which is freely homotopic to $\partial D$ for some point $*$ on $K$. The \textbf{meridional rank} of $K$ is the minimal number of meridians which generate $\pi_1(S^3\setminus N(K))$.
\end{defn}

Of course, the rank of the knot group is a lower bound for the meridional rank, and a lower bound for the rank is (one plus) the minimal number of generators of the Alexander module. More subtle lower bounds are achieved by finding a quotient $G$ of $\pi_1(S^3\setminus N(K))$ which sends the meridians of $K$ to a specified conjugacy class of $G$. The minimal number of elements of this conjugacy class needed to generate $G$ is then a lower bound for $\mu(K)$. For example, although the rank of the symmetric group $S_n$ is two, the number of transpositions needed to generate is $n-1$. See \cite{livingston} for more examples and background.

There are multiple equivalent ways to define the bridge number. Each perspective has its own advantage. 

\begin{defn}
The \textbf{bridge number} of $K$, denoted $\beta(K)$, is the minimal number of minima of $K$, taken over all embeddings in $S^3$.
\end{defn}

\subsubsection{Overpass perspective}
\label{overpass perspective}

Let $D$ be a classical knot diagram. An \textit{overpass} (resp. \textit{underpass}) is a path of $D$ that contains at least one over-crossing (resp. under-crossing), but no under-crossing (resp. over-crossing). Any knot diagram can be bridge decomposed as an alternating sequence of overpasses and underpasses. The \textit{overpass bridge number} $\beta_O(K)$ of $K$ is the minimum number of overpasses over all bridge decompositions of $K$. 

\subsubsection{Height function on the plane perspective}\label{plane perspective}
Let $D$ be a classical knot diagram, and let $h:\mathbb{R}^2\rightarrow \mathbb{R}$ be the standard projection map $h(x,y) = y.$ We define the \textit{$\mathbb{R}^2$-bridge number} $\beta_{\mathbb{R}^2}(D)$ of $D$ to be the number of minima on $D$ for $h$. We then define the \textit{Morse bridge number} $\beta_{\mathbb{R}^2}(K)$ of a classical knot $K$ to be the minimum number $\beta_{\mathbb{R}^2}(D)$ over all classical knot diagrams representing $K.$

\subsubsection{Diagram coloring perspective}\label{wirtinger perspective}

For more details, see \cite{blair2020wirtinger}. Let $D$ be a diagram of $K$. A \textit{$k$-partial coloring} of a diagram $D$ is an assignment of the color blue to a subset of $k$ strands of $D$. Given a $k$-partial coloring of $D$, we can perform a \textit{coloring move} at any crossing where the overstrand and one of the understrands is colored: this colors the other understrand, producing a $(k+1)$-partial coloring. See highlighted crossings in Figure \ref{fig:demonstrate}. If there is a $k$-partial coloring of a diagram $D$ such that a sequence of coloring moves will color all of $D$, then we say that $D$ is \textit{$k$-colorable}. The initially colored strands of the $k$-partial coloring are called the \textit{seeds} of the coloring. The minimum $k$ such that $D$ is $k$-colorable is called the \textit{Wirtinger number} $\omega(D)$ of $D$. The Wirtinger number $\omega(K)$ of a knot $K$ is the smallest value $\omega(D)$ minimized over all diagrams of $K$.

It is well-known that $\beta(K) = \beta_O(K) = \beta_{\mathbb{R}^2}(K)$. The fact that $\beta(K) = \omega(K)$ is the main result of \cite{blair2020wirtinger}.

\subsection{Coxeter groups}
The \textit{Coxeter group} $C(\Delta)$ associated to a finite simple graph $\Delta$ with weighted edges is defined as follows: (1) Each vertex of $\Delta$ corresponds to a generator of $C(\Delta)$. (2) If $s$ is a generator of $C(\Delta),$ the $s^2 = 1.$ (3) If $s$ and $t$ are vertices that are connected by an edge of weight $k,$ then $(st)^k = 1.$ An element conjugate to any of the generators is called a \textit{reflection} and the number of vertices of $\Delta$ is the \textit{rank} of $C(\Delta)$. In particular, a generator is itself a reflection. The following proposition, which was observed in \cite{baader2020coxeter}, will be very useful in computing the meridional rank.

\begin{prop}
Suppose that there is a surjective map from $\pi_1(S^3\setminus N(K))$ to $C(\Delta)$ sending meridians to reflections, where the weight on each edge of $\Delta$ is at least two. Then, the meridional rank of $K$ is bounded below by the rank of $C(\Delta)$.
\end{prop}\label{prop:grouptheoryCoxeter}

\subsection{BBKM knots}\label{sec:bbkm}

In \cite{baader2020coxeter,baader2020bridge} many explicit surjections from classical knot groups to Coxeter groups are realized, with specified conjugacy classes for the meridians to map to, thus providing lower bounds for meridional rank. In these examples, they show that the classical Wirtinger number also equals the Coxeter group rank thus proving the meridional rank conjecture for these knots. We begin by giving brief descriptions of such knots.

\begin{defn}
We will refer to the arborescent knots and twisted knots admitting Coxeter quotients found in \cite{baader2020coxeter,baader2020bridge} as \textit{BBKM knots}.
\end{defn}

For our applications, we need the fact that these knots are constructed by piecing together simple pieces. For the twisted knot case, the building blocks are twist regions and each twist region is associated to two meridians. For the arborescent case, the building blocks are rational tangles, where each rational tangle is associated to either a single Coxeter generator or two distinct Coxeter generators.

\subsubsection{Classical twisted knots}
A classical knot diagram can be checkerboard colored, and there are two associated checkerboard surfaces that can be thought of as a union of disks and twisted bands. One can construct a graph $\Gamma$ from this surface, with one vertex for each disk and one edge for each band, weighted according to the number of signed half-twists in the band.
If a classical knot $K$ admits a diagram with a checkerboard surface such that the edge weights of the induced graph $\Gamma$ are all at least two in absolute value, and such that the plane dual graph $\Gamma^*$ has no multiple edges, then $K$ is called a \textit{twisted knot} \cite{baader2020coxeter}.

\subsubsection{Classical arborescent knots}
Each classical arborescent knot can be encoded with a weighted tree. Each vertex in this weighted tree is an annulus where the weight corresponds to the number of twists. An edge connecting two vertices corresponds to a plumbing of the two annuli. The arborescent knots that are BBKM knots can be further broken into two sub-families: arborescent knots associated to bipartite trees with even weights \cite{baader2020coxeter}, and arborescent knots associated to plane trees whose branching points carry a straight branch to at least three leaves \cite{baader2020bridge}.

\section{Labelling Knots with $p$-Cycles}\label{sec:pcycles}
In this section we develop a procedure to obtain knots colorable by $p$-cycles. 


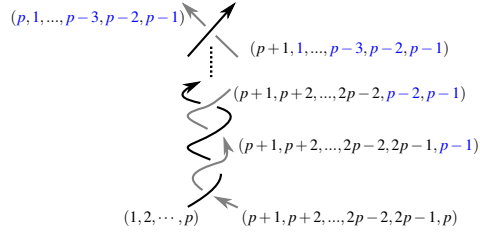
\begin{figure}[ht!]
\labellist
\small\hair 2pt

\pinlabel \tiny\text{$(p+1,p+2,...,2p-2,2p-1,p)$} at 556 -29
\pinlabel \tiny\text{$(p+1,p+2,...,2p-2,2p-1,\textcolor{blue}{p-1})$} at 586 209
\pinlabel \tiny\text{$(p+1,p+2,...,2p-2,\textcolor{blue}{p-2},\textcolor{blue}{p-1})$} at 556 381
\pinlabel \tiny\text{$(p+1,\textcolor{blue}{1},...,\textcolor{blue}{p-3},\textcolor{blue}{p-2},\textcolor{blue}{p-1})$} at 556 531
\pinlabel \tiny\text{$(\textcolor{blue}{p},\textcolor{blue}{1},...,\textcolor{blue}{p-3},\textcolor{blue}{p-2},\textcolor{blue}{p-1})$} at -276 631
\pinlabel \tiny\text{$(1,2,\cdots,p)$} at -66 -29

\endlabellist
 \centering
    \includegraphics[width=.06\textwidth]{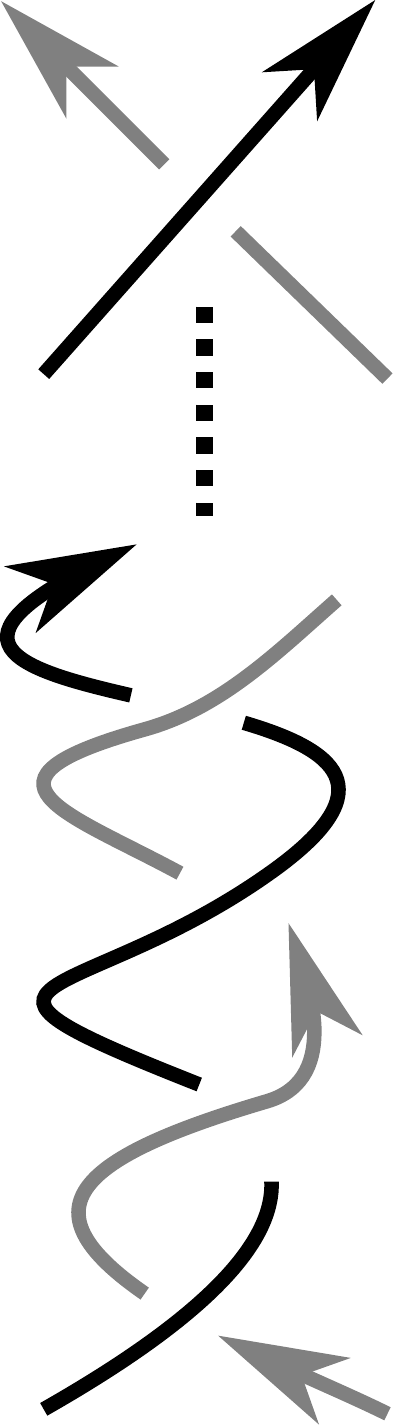}
    \caption{Labeling a twist region with $2p-1$ crossings by overlapping step $p$-cycles.}
    \label{fig:oddlabel}
\end{figure}

When we label knot diagrams with elements of order greater than two, orientations become important. We begin by showing that twist regions can be consistently labelled by $p$-cycles.

\begin{lem}\label{lemma:oddtwist}
The oriented classical two braid with $2p-1$ crossings pictured in Figure \ref{fig:oddlabel} can be labeled with step $p$-cycles $(1,2,\cdots, p)$ and $(p,p+1,\cdots ,2p-1)$.
\end{lem}

\begin{proof}
The tangle has two components (colored black and gray in Figure \ref{fig:oddlabel}). Starting at the bottom right label $(p+1,p+2,...,2p-2,2p-1,p)$, this component will go under the other component $p$ times and the sequence of labelings that will appear in order after the bottom left label is $(p+1,p+2,...,2p-2,2p-1,p-1) \rightarrow (p+1,p+2,...,2p-2,p-2,p-1) \rightarrow (p+2,p+3,...,p-3,p-2,p-1) \rightarrow \cdots$ and will emerge on the top left as $(p,1...,p-1)$.

Similarly, starting at the bottom left label $(1,2,...,p)$, this component will go under the other component $p-1$ times and emerges on the top left as $(p+1,...,2p-2,2p-1,p)$.

If the labels $(1,2,...,p)$ and $(p, p+1, ..., 2p-1)$ are permuted, a similar calculation shows that the top left labelling will agree with bottom left, and top right labelling with the bottom right.
\end{proof}



We now introduce a subset of the set of $p$-cycles which we will use to create more complicated knots with labellings. 
For convenience, we define $G_{p,n}$ as the symmetric group $S_{np-(n-1)}$ when $p\geq 2$ is even, and as the alternating group $A_{np-(n-1)}$ when $p\geq 3$ is odd.

\begin{defn}\label{def:step}
A {\bf step $p$-cycle} is a $p$-cycle of the form $(a,a+1,a+2,\dots,a+p-1)$. Given integers $p,n$, with $p\geq 3,n\geq 2$, we will refer to the set of {\bf once-overlapping step $p$-cycles} $O_{p,n}=\lbrace(1,...,p),(p,...,2p-1),\cdots,$  $((n-1)p-(n-2),...,np-(n-1))\rbrace \subset G_{p,n}$.
\end{defn}


Note that if $x$ and $y$ are in $O_{p,n}$, then they are either once-overlapping, or they permute disjoint elements and therefore commute: $xy=yx$. Note also that if they are once-overlapping, their product is a $(2p-1)$-cycle, e.g.\ $(1,2,3)(3,4,5)=(1,2,3,4,5)$.


Hoping to apply the following result of Annin and Maglione, we also would like to make sure that the elements of $O_{p,n}$ generate $G_{p,n}$.

\begin{thm}\cite[Theorem 3.1]{annin2012economical}
Let $m$ and $r$ be positive integers with $m \geq r \geq 2$ such that $(m, r) \neq
(2, 2)$ and $(m, r) \neq (3, 3)$. If $r$ is odd (respectively, even), then the minimum number
of $r$-cycles needed to generate $A_m$ (respectively, $S_m$) is max$\lbrace 2, \ceil{\frac{m-1}{r-1}} \rbrace.$
\end{thm}

If $m= np-(n-1)=n(p-1)+1$, then $\ceil{\frac{m-1}{p-1}} = \ceil{\frac{np-n}{p-1}} = n.$ Note that knots in $\mathcal{K}_{p,n}$ can be labeled by $n$ $p$-cycles from $G_{p,n}$. We still need to show that these overlapping step cycles generate.

\begin{lem}
The collection of $n$ overlapping step $p$-cycles $O_{p,n}=\lbrace(1,...,p),(p,...,2p-1),(2p-1,....,3p-2),\cdots,$  $((n-1)p-(n-2),...,np-(n-1))\rbrace$ generates $A_{np-(n-1)}$ (respectively, $S_{np-(n-1)}$), when $p$ is odd (respectively, even).\label{generatingset}
\end{lem}

\begin{proof}
By Corollary 2.4 of \cite{annin2012economical}, the set of all step cycles of length $p$ generates $A_{np-(n-1)}$ for $p$ odd (resp. $S_{np-(n-1)}$ for $p$ even), so we need only show that $O_{p,n}$ generates the set of all step $p$-cycles. Note that the (ordered) product of the elements of $O_{p,n}$ is $(1,...,p)(p,...,2p-1)\cdots((n-1)p-(n-2),...,np-(n-1)) = (1,2,\dots,np-(n-1))$, an $(np-(n-1))$-cycle: we denote this element $\tau$. Now we will show that any step $p$-cycle is generated by $\tau$ and $(1,2,\dots,p)$.
Let $a$ be an integer in $\{1,2,\dots,np-(n-1)\}$. Observe that the step cycle of length $p$ starting at $a$ can be written as:

\begin{align*}
    \tau^{a-1}(1,2,...,p)\tau^{-(a-1)} = (\tau^{a-1}(1),\tau^{a-1}(2), \dots, \tau^{a-1}(p)) = (a,a+1,\dots,a+p-1) .
\end{align*}
\end{proof}


Brunner described a method to construct a link in $S^3$ from a weighted simple planar graph, which admits a surjection to a Coxeter group whose presentation can be read off from the graph.
Now we adapt his construction to the more subtle case of labeling by step $p$-cycles as opposed to just order two elements. The following construction builds links in $S^3$ which can be labelled by overlapping $p$-cycles, yielding a surjection of the link group to $G_{p,n}$ which detects the meridional rank.

\begin{con}\label{const:oddp}
Let $\Gamma$ be a simple planar graph, and form its dual $\Gamma^*$. Label the edges of $\Gamma^*$ with elements from the set $\{0,1\}$.
As in \cite{brunner}, blow up the vertices of $\Gamma^*$ to disks, and the edges to twisted bands. At this stage, the number of half-twists in each band is only determined up to parity: an even number of half-twists if that edge was labelled with 0, and an odd number if labelled with 1. This determines the connectivity of the resulting link diagram. Choose an orientation on each component. Let $n=|V(\Gamma)|$. As shown in \cite{brunner}, there is a choice of $n$ meridians, one for each region of $\mathbb{R}^2 - \Gamma^*$, which will generate the group of the complement of the resulting link. The following claim summarizes a strategy to choose the number of half-twists in each box so that the resulting link $L$ will be colorable by step $p$-cycles.

\vspace{.2cm}

\noindent{\bf Claim.} If a generating set of $n$ meridians on the resulting link diagram can be labelled by the $n$ elements of ${O_{p,n}}$, subject to the following rules, then the labelling corresponds to a surjection $\pi_1(S^3\setminus L) \rightarrow G_{p,n}$, sending meridians to step $p$-cycles. 

{\bf Rule 1.} If the labels at the bottom of a twist box are the same or permute disjoint cycles, then any number of half-twists may be chosen (respecting the parity choice already made).

{\bf Rule 2.} If the labels at the bottom of a twist box are once-overlapping, then the number of half-twists can be chosen as follows (see Figure \ref{fig:twistboxes}):

\hspace{\parindent} {\it Case I.} If both strands travel the same direction, then any multiple of $2p-1$ half-twists may be chosen (respecting the parity choice already made).

\hspace{\parindent} {\it Case II.} If one strand travels up and one down, then any even multiple of $2p-1$ half-twists may be chosen (note this necessitates having an even twist box to begin with).
\qed

\vspace{.3cm}
We now prove that the rules above yield a valid labelling of the diagram, which amounts to showing that the labels at the bottoms of the twist boxes agree with the corresponding labels at the top. We prove the case $k=1$, which implies the statement for all positive $k$. The proof is similar for negative $k$, and trivial when $k=0$.

Let $x,y \in O_{p,n}$. If $x$ and $y$ do not overlap, then they commute, and the Wirtinger relations become trivial. Similarly, if $x=y$ the relations are trivial. Otherwise, $x$ and $y$ are once-overlapping, so $xy$ and $\overline{x}y$ are $(2p-1)$-cycles. 

First consider Case I, the left-hand picture of Figure \ref{fig:twistboxes}. The calculation in Figure \ref{fig:oddlabel} shows that the top-left label is $(1,...,p)$ and the top-right label is $(p,...,2p-1)$. In terms of $x_1$ and $y_1$, this says that 
$x_1^\prime=({x_1}y_1)^{-(p-1)}y_1({x_1}y_1)^{p-1}=x_1$, and\\
$y_1^\prime=({x_1}y_1)^{-p}x_1({x_1}y_1)^{p}=y_1$ (in fact, the first equation implies the second, since $({x_1}y_1)^{2p-1}$ is trivial).
As noted in the proof of Lemma \ref{lemma:oddtwist}, the desired conclusion holds if these labels are permuted.

Now consider Case II, the right-hand picture of Figure \ref{fig:twistboxes}. The Wirtinger relations from the crossings in the twist-box imply that $x_2^\prime=(\overline{x_2}y_2)^{-(2p-1)}x_2(\overline{x_2}y_2)^{2p-1}$. Since $\overline{x_2}y_2$ is a $(2p-1)$-cycle, we have $x_2^\prime=x_2$. Similarly, $y_2^\prime=(\overline{x_2}y_2)^{-(2p-1)}y_2(\overline{x_2}y_2)^{2p-1}=y_2$.

\end{con}


\begin{defn}\label{def:KPN} 
Let $p,n\in\mathbb{Z}$ with $p,n\geq 2$. We define $\mathcal{K}_{p,n}$ to be the set of knots resulting from the above construction. 


\begin{figure}
    \centering
    \includegraphics[width=0.3\linewidth]{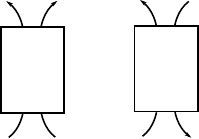}
        \put(-101,36){\tiny{$k(2p-1)$}} 
        \put(-111,0){\tiny{$y_1$}}
        \put(-74,0){\tiny{$x_1$}}
        \put(-111,75){\tiny{$y_1^\prime$}}
        \put(-74,75){\tiny{$x_1^\prime$}}
        \put(-30,36){\tiny{$2k(2p-1)$}} 
        \put(-38,0){\tiny{$y_2$}}
        \put(-2,0){\tiny{$x_2$}}
        \put(-38,75){\tiny{$y_2^\prime$}}
        \put(-2,75){\tiny{$x_2^\prime$}}
    \caption{The allowable twist regions in Construction \ref{const:oddp}. The relations $x_i=x_i^\prime$ and $y_i=y_i^\prime$ are guaranteed whenever $x_i$ and $y_i$ are chosen from $O_{p,n}$.}
    \label{fig:twistboxes}
\end{figure}
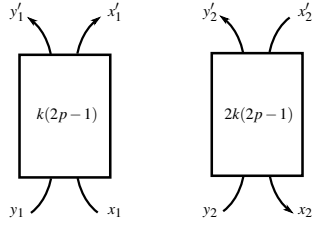

\end{defn}

\begin{exmp}
$K_{2,2}$ is the set of 2-bridge twisted knots with a surjection to $S_3$, the third symmetric (and dihedral) group. It is straightforward to check that when $|V(\Gamma)|=2$, any knot resulting from Construction \ref{const:oddp} will be of the form $T(2,2m+1)$. Thus $K_{2,2}=\{T(2,6k+3):k\in\Z\}$, the 2-bridge torus knots which admit a tricoloring.
\end{exmp}

\begin{exmp}
The torus knot $T(2,2p-1)$ is in $K_{p,2}$, and therefore admits a surjection to $G_{p,2}$, sending meridians to $p$-cycles. The $n$-fold connected sum of $T(2,2p-1)$ is in $\mathcal{K}_{p,n+1}$.
\end{exmp}

\begin{rem} \label{rem:manycolorings}
Note that Construction \ref{const:oddp} shows that a single knot may admit arbitrarily many surjections to symmetric or alternating groups of different orders, sending meridians to cycles of different lengths. For example, consider the torus knot $T(2,35)=K$. Since $35 = 2*18-1$, $K$ may be labelled by 18-cycles in $G_{18,2}=S_{35}$.
Since $35=7*5=7(2*3-1)=5(2*4-1)$, $K$ may also be labelled by 3-cycles in $G_{3,2}=A_5$, and by 4-cycles in $G_{4,2}=S_7$.
\end{rem}

\begin{exmp} 
The generalized ($q_1,q_2,q_3,...,q_n)$-pretzel knot where $q_i$ is an odd multiple of $2p-1$ for $i\neq n$ and $q_n$ is even is in $\mathcal{K}_{p,n}$ (note that the first and last labels commute, so the crossings in the even twist region do not add any relations). Figure \ref{fig:example} depicts a specific case where our knot $J$ is the $(5,5,5,2)$-pretzel knot. 
\end{exmp}
\begin{figure}
\labellist
\small\hair 2pt

\pinlabel \tiny\text{$(123)$} at 2570 240
\pinlabel \tiny\text{$(345)$} at 1400 239
\pinlabel \tiny\text{$(567)$} at 1770 265
\pinlabel \tiny\text{$(789)$} at 2200 385

\endlabellist
    \centering
   \includegraphics[width=6cm]{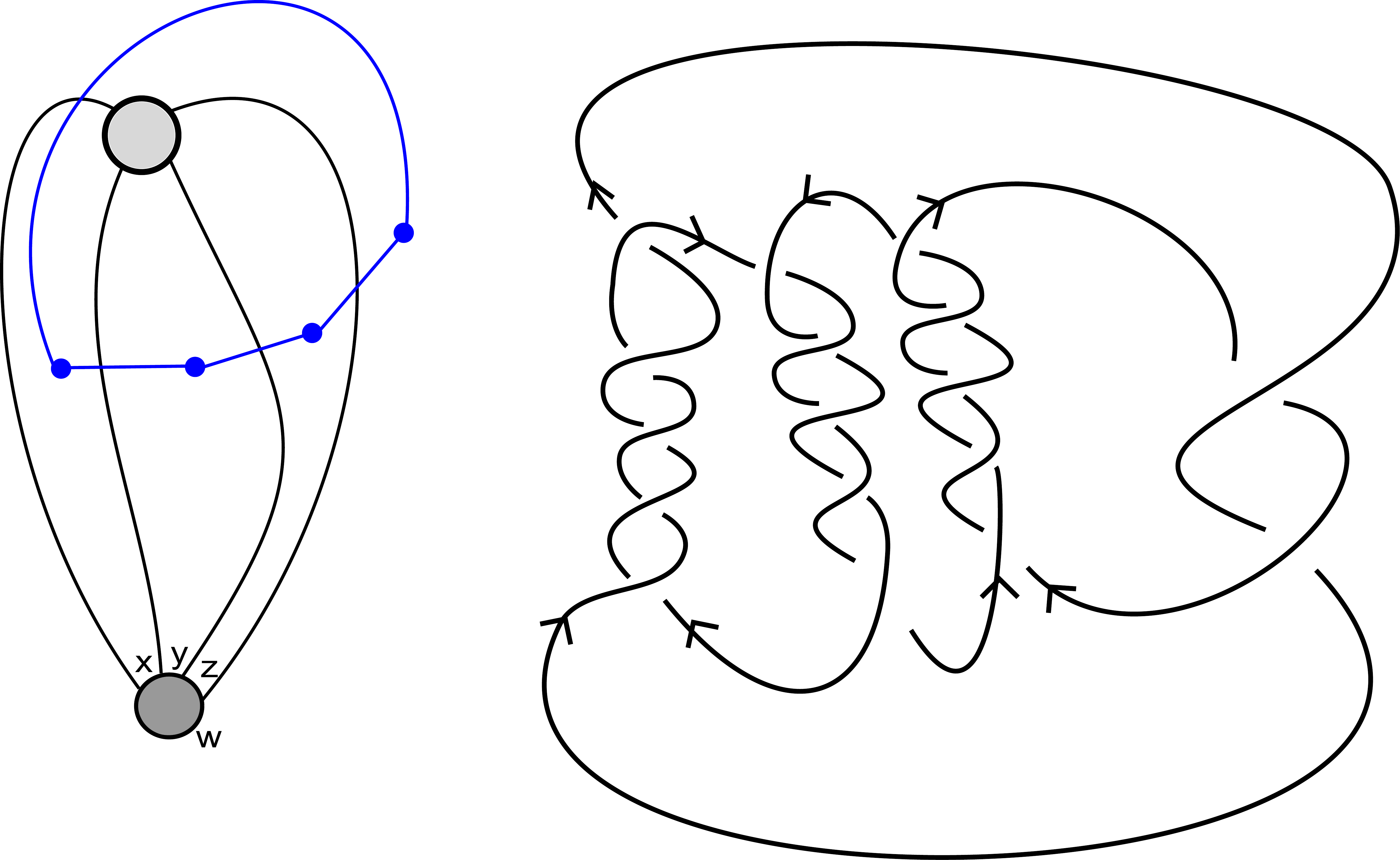} 
    \caption{At left, the cycle graph $\Gamma$ and its dual $\Gamma^*$ are shown. The vertices of $\Gamma^*$ are fattened into disks and edges into twisted bands, forming the generalized pretzel knot $K \in \mathcal{K}_{3,4}$ (right).} 
    \label{fig:example}%
\end{figure}
\begin{figure}[ht!]

 \centering
    \includegraphics[width=.3\textwidth]{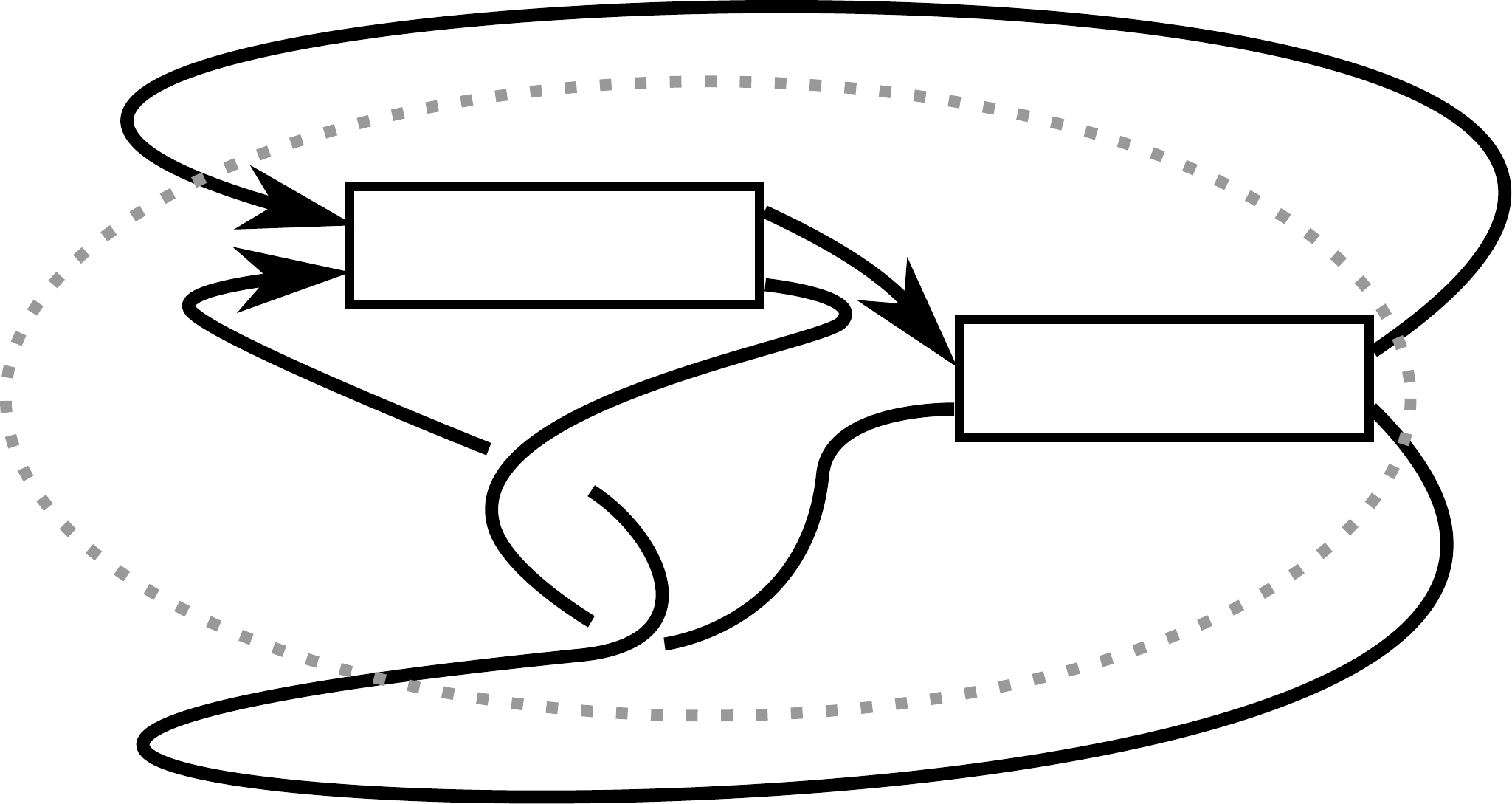}
    \caption{The closure of a rational tangle in $\mathcal{K}_{p,2}$. Each rectangle represents a horizontal twist region containing a multiple of $2p-1$ crossings (one odd and one even).}
    \label{fig:2bridgekpn}
\end{figure}
\begin{rem}
Instead of using just twist regions in the construction of $\mathcal{K}_{p,n}$, one may use non-integer rational tangles by making sure that the end result is a knot and the orientations are compatible (see Figure \ref{fig:2bridgekpn}).
\end{rem}

The preceding discussion proves that for any knot $K$ in $\mathcal{K}_{p,n}$, the meridional rank of $K$ is equal to $n$. The same techniques utilizing the Wirtinger number as in \cite{blair2020wirtinger} show that their bridge number is also $n$.

\begin{cor}
If $K\in \mathcal{K}_{p,n}$, then $\mu(K)=\beta(K)=n$.
\end{cor}

\begin{prop} \label{prop:kpn_additive}
This construction is well-suited for connected sums of knots. If $K_1\in K_{p,n_1}$ and $K_2\in K_{p,n_2}$, then $K_1 \# K_2 \in K_{p,n_1+n_2-1}$.
\end{prop}

\section{Meridional Rank and Bridge Number of 2-Knots}\label{sec:2-knots}
In this section, we define the meridional rank and bridge number of knotted surfaces in $S^4$, and use the labelings from the previous section to prove Theorem \ref{thm:twistspunrank}. We then investigate the additivity of meridional rank under connected sum of 2-knots and prove Theorem \ref{thm:nonadditivity}.

\subsection{2-knots in $S^4$}
A {\bf 2-knot} $K:S^2\rightarrow S^4$ is a smoothly embedded sphere in $S^4$. We assume our embeddings are in Morse position with respect to the standard height function $S^4\rightarrow \R$, and therefore have finitely many critical points.

The {\bf bridge number} and {\bf meridional rank} of a 2-knot are defined exactly analogously as in the case of knots in $S^3$: the bridge number is the minimal number of minima over all Morse embeddings, and the meridional rank is the minimal number of meridians needed to generate $\pi_1(S^4\setminus N(K))$. As in the case of classical knots, the meridional rank of a 2-knot is a natural lower bound for its bridge number. Sometimes we will write $\pi K$ as an abbreviation for $\pi_1(S^4\setminus K)$.


\begin{prop}
Let $K$ be a 2-knot. Then $\mu(K)\leq \beta(K)$.\label{prop:merandbridge}
\end{prop}

\begin{proof}
Consider an embedding $K$ with $\beta(K)=n$ minima. Taking a meridian $\mu_i$ to each of these minima yields a generating set for $\pi_1(S^4\setminus N(K))$, since the index 0 critical points of $K$ correspond to the 1-handles in a handle decomposition of $S^4\setminus N(K)$ \cite{gompfandstipsicz}.
\end{proof}

\subsubsection{Twist-spun knots}\label{definezeeman}
Let $K\subseteq S^3$ be a knot. Delete a small neighborhood of a point on $K$ to obtain a tangle $(B^3,K^\circ)$. Let $(B^4,D)$ denote the trace of the identity isotopy of $K^\circ$ in $B^3$, i.e.\ $(B^4,D)=(B^3,K^\circ)\times I$ and $D$ is the standard half-spun disk for $K\# {-K}$. Let $(B^4,D_m)$ denote the trace of the isotopy that rotates $K^\circ$ around its axis $m$ times, for some $m\in\Z$. The \textbf{$m$-twist spin} of $K$ is defined to be the 2-knot $(S^4,\tau^m K) = (B^4,D) \cup (B^4,D_m)$. 

Twist-spun knots were introduced by Zeeman \cite{zeeman}, who proved that for $m\neq 0$, \newline $S^4\setminus N(\tau^m K)$ is fibered by the $m$-fold cyclic branched cover of $K$. Note this implies that $\tau^{\pm 1}K$ is always unknotted. The group of $\tau^m K$ is a quotient of $\pi K$, obtained by centralizing the $m^{\text{th}}$ power of a meridian. More generally, we can consider a general ambient isotopy $f=\{f_t \ | \ t\in [0,1]\}$ of the tangle, where $f|_{\partial B^3}=id, f_1(K^\circ)=K^\circ$, as in \cite{litherland1979deforming}. The motion $f$ is called a deformation and $fK$ is called a {\bf deform-spin} of $K$.

\begin{prop}\label{prop:easyupperbound}
Let $K$ be a knot in $S^3$. Then, $\beta(f K) \leq \beta (K)$ and $\mu(f K) \leq \mu (K)$.
\end{prop}

\begin{proof}
The deform-spun knot $fK$ can be thought of as only doing the deformation in the top half, i.e.\ we glue together the deform-spun disk on top and the standard half-spun ribbon disk for $K\# -K$ on bottom. By starting with $K$ in minimal bridge position, and removing a small 3-ball centered on a maximum of $K$, the half-spun disk will have $\beta(K)$ minima. The top half has no minima, as it is a ribbon disk for $K\# -K$. Thus $\beta( fK)\leq \beta(K)$.

The group $\pi (f K)$ is obtained from $\pi K$ by identifying $f_1(x)$ with $x$, for each meridian $x$ of $\pi K$ \cite{litherland1979deforming}. Thus there is a quotient map $\pi K \to \pi (f K)$, which sends meridians to meridians because meridional curves of the punctured $K$ are meridional curves of $f K$.
This shows the second inequality.

\end{proof}



\subsection{Twist-spun 2-knots}
Using Construction \ref{const:oddp}, for any twist index $m$, we can construct examples of $m$-twist-spun 2-knots whose meridional rank and bridge number are detected.


\twistspun*

\begin{proof}[Proof of Theorem \ref{thm:twistspunrank}]

We take $K$ from the family $\mathcal{K}_{m,n}$ defined in Section \ref{sec:pcycles}. These knots have surjections to $G_{m,n}$, sending meridians to step $m$-cycles. 
Since the group of the twist-spun knot centralizes $m^\text{th}$ powers of meridians, the quotient map $\pi K \to G_{m,n}$ factors through the map $\pi K \to \pi (\tau^m K)$, so $\pi (\tau^m K)$ has a surjection to $G_{m,n}$ as well, sending meridians to $m$-cycles. As shown in Section \ref{sec:pcycles}, $n$ $m$-cycles are needed to generate $G_{m,n}$, and therefore $\mu(\tau^m K) \geq n$. 

The bridge number of $\tau^m K$ is at most $\beta(K)=n$, by Proposition \ref{prop:easyupperbound}. Thus $\mu(\tau^m K) = \beta(\tau^m K) = n$.

\end{proof}

\begin{rem}
If we combine Theorem \ref{thm:twistspunrank} with the observation in Remark \ref{rem:manycolorings}, we can actually show that for any $n\geq 2$ and vector $(m_1,\dots,m_q)$ with $|m_i| \neq 1$, there exist infinitely many knots $K\subset S^3$ such that $\mu(\tau^{m_i} K) = \beta(\tau^{m_i} K) = n$ for all $i$. For example, for odd $j \geq 1$ let $K_j$ be the $n$-fold connect sum of $T(2,jM)$, where $M=(2m_1 - 1)(2m_2 - 1)\cdots(2m_q-1)$.
\end{rem}

We remark that Coxeter quotients, where images of meridians have order two, are sufficient to prove the theorem for all even twist indices, and thus any BBKM knot will suffice when $m$ is even. This observation was the starting point of our efforts to find compatible quotients for any $m$-twist-spun knot.

In addition, the technique used in Theorem \ref{thm:twistspunrank} can be used for other motions that affect the fundamental group of $K$ by centralizing powers of a certain element $\gamma$ of $\pi K$. To elaborate, we remind the reader of the proof of Proposition \ref{prop:easyupperbound}. Given a meridional presentation $\langle x_1,...,x_n \;\ | \;\  R\rangle$ of $K$, we get a meridional presentation $\langle x_1,...,x_n \;\ | \;\ R,f_1(x_i) = x_i \;\  (i=1,...,n)\rangle$ for the deform spin of $K$ with motion $f$. Note that $\gamma$ is a meridian when the motion is $m$-twist-spinning and $f_1(x_i)=x^{-m}x_ix^m$ for a meridian $x$. When the motion is Litherland's $m$-roll spinning \cite{litherland1979deforming}, the element $\gamma$ is the Seifert longitude $\lambda$ and $f_1(x_i)=\lambda^{-m}x_i\lambda^m$.

\begin{cor}
    Let $K \in \mathcal{K}_{p,n}$. Using the notations in the previous paragraph, suppose that $f_1(x_i) = \gamma^{-1}x_i\gamma$. Then, $\mu(K) = \mu(f^m K)= \beta(f^m K) = n$, where $m = |G_{p,n}| = \frac{1}{2}(np-(n-1))!$ when $p$ is odd and $(np-(n-1))!$ when $p$ is even.
\end{cor}
\begin{proof}
    Recall that $K \in \mathcal{K}_{p,n}$ has a surjection to $G_{p,n}$. Call that surjection $\phi$. We get a surjective homomorphism from  $\pi(f^{m}K) = \langle x_1,...,x_n \;\ | \;\ R,\gamma^{-m}x_i\gamma^{m} = x_i \;\  (i=1,...,n)\rangle$ to $G_{p,n}$ as well because $\phi(\gamma)$ raised to the order of $G_{p,n}$ is trivial.
\end{proof}

\subsection{Behavior of meridional rank under connected sum}\label{sec:behaviorcsum}

In this section we study how the meridional rank of 2-knots (knotted spheres in $S^4$), changes under connected sum. Note that if the meridional rank conjecture for classical knots is true then their meridional ranks must be $(-1)$-additive, since the bridge number is known to have this property. One elementary bound for the meridional rank of a connected sum is the following.

\begin{prop} \label{prop:connectsum}
For any orientable knotted surfaces $K_1,K_2$:
\begin{center}$max\{\mu(K_1), \mu(K_2)\} \leq \mu(K_1\# K_2) \leq \mu(K_1) + \mu(K_2) - 1$.\end{center}
\end{prop}

\begin{proof}
The group of $K_1\# K_2$ surjects onto the group of $K_i$ by abelianizing the other factor. So if $x_1,\dots,x_n$ are meridians which generate $\pi(K_1\# K_2)$, then their images under this quotient map are meridians which generate $\pi(K_i)$. This proves the first inequality. The second is proved by taking the obvious presentation for $\pi(K_1\# K_2)$: a minimal set of meridional generators for each factor, and then identifying a meridian of $K_1$ with one of $K_2$ to form the amalgamated product.

\end{proof}

Theorem \ref{thm:nonadditivity} below proves that the meridional rank of a connected sum of $n$ 2-knots can achieve any value in between the theoretical bounds given by Proposition \ref{prop:connectsum}. Thus either the bridge number also fails to be $(-1)$-additive for these examples, or the meridional ranks of these knotted spheres are strictly less than their bridge numbers.

\nonadditivity*

\nonaddcorollary*

The lemma below is one of the more striking special cases of the theorem, from which the other cases are obtained by taking connected sums wisely.

\begin{lem}\label{lem:nonadditivity}
Let $n, m_1,\dots,m_n\in\mathbb{N}$ such that $m_i\geq 2$ are relatively prime. Let $k_1,\dots,k_n$ be 2-bridge knots, and let $K_i = \tau^{m_i} k_i$. Then $\mu(K_i)=2=\mu(K_1\#\cdots\# K_n)$.
\end{lem}

\begin{proof}
Let $\langle x_i, y_i | r_i \rangle$ be a Wirtinger presentation for $\pi k_i$, where $x_i$ and $y_i$ are meridians of $k_i$. Then a presentation for $K_i$ is obtained by adding the relation $[x_i^{m_i},y_i]=1$, or equivalently $x_i^{m_i}=y_i^{m_i}$. It is convenient to think of the group of the connected sum $K_1\# \cdots \# K_n$ as being amalgamated in a ``zig-zag'' fashion from the groups of the summands $K_i$: for $i$ odd we amalgamate $x_i$ with $x_{i+1}$, and for $i$ even we amalgamate $y_i$ with $y_{i+1}$. Then a presentation for the group of the connected sum is $\langle x_1, y_1, \dots, x_n, y_n | r_1, \dots, r_n, x_1^{m_1}=y_1^{m_1},\dots ,x_n^{m_n}=y_n^{m_n}, x_1=x_2, y_2=y_3, \dots , z_{n-1}=z_n\rangle$, where the symbol $z$ stands for $x$ if $n$ is even and $y$ if $n$ is odd.

Now we prove by induction that for $n$ even, the meridians $y_1$ and $y_n$ generate the group of $K=K_1\#\cdots\# K_n$, and that for $n$ odd, $y_1$ and $x_n$ generate the group of $K$. Since each of the groups $\pi K_i$ inject into the group of $K$, $\pi K$ is not cyclic, so we will conclude that $\mu(K)=2$.

Specifically, we will show that for $n$ even, $x_n=x_{n-1}$ is in the subgroup generated by $y_1$ and $y_n$, and for $n$ odd, $y_n=y_{n-1}$ is in the subgroup generated by $y_1$ and $x_n$. Then, by induction, the stated pairs of elements generate the group of $K$.

When $n=1$, the group of $K=K_1$ is generated by $x_1$ and $y_1$ by assumption. Now assume the claim is true for less than or equal to $n-1$ summands, and consider the case of $K=K_1\#\cdots\# K_n$. 

Let $n$ be even, $M=m_1 m_2 \cdots m_{n-1}$, and note that the twist-spin relations imply that $y_1^M=x_{n-1}^M=x_n^M$. Since $M$ is relatively prime to $m_n$, there exist integers $a$ and $b$ so that $aM+bm_n=1$. Then $y_1^{aM}y_n^{bm_n}=x_n^{aM}x_n^{bm_n}=x_n=x_{n-1}$, so this element is in $\langle y_1,y_n\rangle$. 

Now let $n$ be odd and consider $M,a,b$ as before. Then $y_1^M=y_{n-1}^M=y_n^M$, and $y_1^{aM}x_n^{bm_n}=y_n^{aM}y_n^{bm_n}=y_n=y_{n-1}$. Thus $y_1$ and $x_n$ generate $y_n=y_{n-1}$, as claimed. 


\end{proof}

In \cite{kanenobu1996weak}, Kanenobu defined a family of 2-knots in order to prove an analogous statement to Theorem \ref{thm:nonadditivity} regarding the \textit{weak unknotting number} of a 2-knot, the fewest number of meridian-identifying relations which abelianize the knot group.

For the lower bound in Theorem \ref{thm:nonadditivity}, we will again make use of Construction \ref{const:oddp}. Recall that this construction is $(-1)$-additive under connected sum (Proposition \ref{prop:kpn_additive}). For a 2-knot $K$, let $\alpha J$ denote the connected sum of $\alpha$ copies of $J$. 



\begin{defn} \label{def:kanenobu}
Let $p_1,\dots,p_n,q\geq1$ such that $max\{ p_i\} \leq q \leq \sum p_i$ and $p_1\geq p_2 \geq \dots \geq p_n$, and choose $j$ such that $\displaystyle{\sum_{i=1}^{j-1} (p_i - 1) \leq q - 1 \leq \sum_{i=1}^{j} (p_i - 1)}$. Let $m_1,\dots,m_n$ be relatively prime integers with $|m_i|\geq2$, and let $T_i=\tau^{m_i} k_i$, where $k_i\in K_{m_i,2}$, e.g.\ $k_i=T(2,2m_i-1)$. Now define

\[ K_i=\begin{cases} 
      (p_i - 1) T_1 & i< j \\
      \left(q + j - 2 -(p_1+\cdots+ p_{j-1}) \right)T_1 \# (p_j - 1) T_j
       & i=j \\
      (p_i - 1) T_i & i>j 
   \end{cases}
\]

and let $K=K_1\#\cdots\# K_n$.
\end{defn}





\begin{proof}[Proof of Theorem \ref{thm:nonadditivity}]

Given $p_1,\dots,p_n,q$ as in the statement of Theorem \ref{thm:nonadditivity}, choose a family $K_1, \dots, K_n$ as in Definition \ref{def:kanenobu}. 

Note that each $T_i$ has meridional rank 2, since $\pi(\tau^{m_i} k_i)$ has a surjection to $G_{m_i,2}$. Also $K = (q-1) T_1 \# (p_j - 1) T_j \# \cdots \# (p_n - 1) T_n$, so $\mu(K)\geq \mu((q-1) T_1) = q$.

Following Kanenobu, note that $K$ can also be written as $(q-p_j)T_1 \# (p_j-p_{j+1})(T_1\# T_j) \#$ $ (p_{j+1}-p_{j+2})(T_1 \# T_j \# T_{j+1}) \# \cdots \# (p_{n-1}-p_n)(T_1 \# T_j\#\cdots \# T_{n-1}) \# (p_n-1)(T_1\#T_j\#\cdots\# T_n)$. Each of the 2-knots $T_1\#T_j\#\cdots\# T_{j+i}$ has meridional rank 2 by Lemma \ref{lem:nonadditivity}, hence $\alpha(T_1\#T_j\#\cdots\# T_{j+i})$ has meridional rank at most $\alpha+1$. Then $K$ is a connected sum of $q-1$ 2-knots, each of meridional rank 2, so $\mu(K)$ is at most $q$, by repeated application of Proposition \ref{prop:connectsum}.


\end{proof}


\begin{rem}
The lower bound used by Kanenobu in \cite{kanenobu1996weak} is the {\em Nakanishi index}, the minimal number of generators of the Alexander module. This approach would also work here, but we prefer to use the $p$-cycle colorings developed in Construction \ref{const:oddp}, to be self-contained and also because this readily yields infinitely many families of 2-knots (for fixed parameters $p_1,\dots,p_n,q$) satisfying the theorem. 
\end{rem}

\begin{rem}
The weak unknotting number studied by Kanenobu in \cite{kanenobu1996weak} is a natural lower bound for the \textit{stabilization number}, the minimal number of 1-handle stabilizations needed to produce an unknotted surface. In \cite{JKRS}, a theorem analogous to Kanenobu's theorem and Theorem \ref{thm:nonadditivity} is proved for the the (algebraic) \textit{Casson-Whitney number} of a 2-knot, a measure of complexity regarding regular homotopies to the unknot, using the same examples. Theorem \ref{thm:nonadditivity} reproves both of these theorems, since $\mu(K)-1$ is an upper bound for these algebraic unknotting numbers. 
\end{rem}


\begin{rem}
If one starts with a connected sum as in Lemma \ref{lem:nonadditivity} with at least 2 factors and performs a stabilization to effect Kanenobu's relation, a torus $S$ with group $\Z$ is obtained; Kanenobu asks if this torus is smoothly unknotted \cite{kanenobu1996weak}. In \cite{JKRS} it is shown that $S\#T = T\# T$, where $T$ is an unknotted torus. If $S$ is smoothly knotted, and if $\beta(S)>1$, then $\beta(S)>\beta(S\# T)$, and this would be an example of bridge number collapsing. This is purely conjectural, as current tools have not been able to identify any smoothly knotted torus (or any orientable surface) with group $\Z$, and if one was identified it is also not obvious how to show that its bridge number is at least 2.
\end{rem}

One final observation is that if the bridge number does collapse in these cases, then one can find an entire Wirtinger presentation with a smaller than expected number of generators. It is not clear that such a presentation exists, as direct substitution of the smaller generating set found in Lemma \ref{lem:nonadditivity} does not yield Wirtinger relations.

\begin{quest}
Let $K$ be as in Theorem \ref{thm:nonadditivity}, with $\mu(K)=q<\sum p_i - (n-1)$. Does $K$ have a Wirtinger presentation with $q$ meridional generators?
\end{quest}

\begin{conj}
The examples in Theorem \ref{thm:nonadditivity} are counterexamples to the MRC for knotted spheres.
\end{conj}

\subsection{A lower bound from the rank of the commutator subgroup}\label{sec:rank of the commutator subgroup}
As proven in the previous section, meridional rank can behave erratically under connected sum, even staying bounded in a connected sum with arbitrarily many summands. In contrast, here we find a lower bound for the meridional rank of a connected sum of twist-spun knots, which shows that if the twist indices in a family are bounded, the meridional rank of increasingly long connected sums must increase asymptotically. The proof is inspired by the well-known argument that $\mu(K)-1$ elements always generate the Alexander module (see e.g.\ \cite{rolfsen}).

\begin{restatable}{thm}{commutatorsubgroup} \label{thm:commutator_subgroup}
Let $\{K_i\}_{i=1}^\infty$ be a collection of twist-spun 2-knots: $K_i=\tau^{m_i} k_i$ for nontrivial classical knots $k_i$, with $|m_i|\geq 2$. If $\{m_i\}_{i=1}^\infty$ is bounded, then $\displaystyle{\lim_{n\to \infty}\mu(K_1\#\cdots \# K_n)=\infty}$.
\end{restatable}

The lower bound comes from the rank of the commutator subgroup of the group of an $m$ twist-spun knot, and the property that conjugation by a meridian has order at most $m$. Zeeman proved that $\tau^{m_i} k_i$ is fibered by $\Sigma_m k$, the $m$-fold cyclic cover of $S^3$ branched over the knot $k_i$ \cite{zeeman}.

\begin{lem}
Let $\{K_i\}_{i=1}^\infty$ be a collection of twist-spun 2-knots: $K_i=\tau^{m_i} k_i$ for nontrivial classical knots $k_i$, with $|m_i|\geq 2$, $n\geq 1$. Let $K=K_1\#\cdots \# K_n$, $M=\mathrm{lcm}(m_1,\dots,m_n)$, and $\displaystyle{N=\sum_{i=1}^n rk(\pi_1 (\Sigma_{m_i} k_i))}
$. Then $\mu(K)\geq 1+N/M$.
\end{lem}


\begin{proof}
For clarity, we first prove the statement in the case that the number of summands $n$ is equal to $1$, i.e.\ that $K=\tau^m k$ and $\mu(K)\geq 1 + \frac{rk(\pi_1(\Sigma_m k))}{m}$.

Suppose that meridians $x_1,...,x_k$ generate $G = \pi_1(S^4\backslash K)$. Denote by $C$ the commutator subgroup of $\pi_1(S^4\backslash K)$. Since $\tau^m k$ is fibered by $\Sigma_m k$, $C\cong\pi_1(\Sigma_m k)$. Recall that for any orientable surface knot group $G$ with commutator subgroup $C$, the abelianization short exact sequence $1\rightarrow C \rightarrow G \rightarrow \Z \rightarrow 1$ is split-exact: a splitting is provided by sending $1\in\Z$ to a meridian of $G$. So we can regard $G$ as a semidirect product: $G\cong C \rtimes \langle x_1 \rangle$. Then each $x_i=x_1 c_i$, for some $c_i\in C$.

Now, letting $x$ denote $x_1$, we have $G=\langle x, xc_2, \dots, xc_n\rangle = \langle x, c_2, \dots, c_n\rangle$. Let $c\in C$. Then $c=w(x,c_2,\dots,c_n)$, i.e.\ $c$ can be written as a word in these generators. Notice that this word must have an exponent sum of zero for all of its $x$ terms, since otherwise it is nontrivial in the abelianization. This means that $c=w^\prime(x^{-j}c_ix^j:j\geq 0, 2\leq i\leq k)$, however since $x^m$ is central in $G$, we need only consider $0\leq j \leq m-1$. Therefore $C=\langle x^{-j}c_ix^j:0\leq j\leq m-1, 2\leq i \leq k \rangle$, and $rk(C)\leq m(k-1)$. Taking $k$ to be minimal and rearranging, we get the desired inequality.



For the general case, the adaptation is to replace $m$ with $M=\mathrm{lcm}\{m_1,\dots,m_n\}$. Note that the commutator subgroup of the knot group of a connected sum is the free product of the individual commutator subgroups. Therefore $C\cong C_1*\cdots * C_n$, where $C_i\cong \Sigma_{m_i}k_i$ is the commutator subgroup of $K_i=\tau^{m_i}k_i$. If $c\in C$ is nontrivial, then $c$ can be written as a product $c=z_1\cdots z_\ell$, where each $z_j$ is nontrivial and in exactly one of the $C_i$'s. Then $x^{-M}cx^M = c$, because $x^{-M}z_j x^M =z_j$ for each $j$. Following the previous argument, $C=\langle x^{-j}c_ix^j:0\leq j\leq M-1, 2\leq i \leq k \rangle$, and $rk(C)= rk(C_1)+\dots+rk(C_n)\leq M(k-1)$. Since $rk(\pi_1(\Sigma_{m_i} k_i))\geq 1$, we see that $\mu(K)\geq 1+N/M$, which approaches infinity with $n$ when e.g.\ $M$ is finite.
\end{proof}

The proof of the theorem then follows from noticing that $N\geq n$, since $\pi K_i\ncong \Z$ and therefore $C_i\neq 1$. This proves that to exhibit the extreme behavior in Lemma \ref{lem:nonadditivity}, it was really necessary for the twist-indices to become arbitrarily large.

Weidmann proved that for a connected sum of $n$ nontrivial classical knots, the rank of the knot group, and therefore the meridional rank, is at least $n+1$ \cite{weidmann}. An application of Theorem \ref{thm:nonadditivity} to $\tau^m(k_1\#\cdots\# k_n)=\tau^mk_1\#\cdots\# \tau^m k_n$ yields the following corollary, which gives an analogous version for the twist spin of a connected sum.


\begin{cor} \label{cor:weidmann}
Let $k_1,\dots,k_n$ be nontrivial classical knots and $|m|\geq2$. Then \begin{center}$1+n/m\leq \mu(\tau^m(k_1\#\cdots\# k_n))\leq \mu(k_1\#\cdots\#k_n)$.
\end{center}
\end{cor}



\section{Welded Knots and Ribbon Tori}\label{sec:weldedmrkconj}

In this section, we investigate the meridional rank and bridge number of ribbon $T^2$-knots. We begin by reminding the readers of a convenient way to represent ribbon tori with generalized knot diagrams. A small circle around a double point is called a {\bf virtual crossing}.

\subsection{Background on virtual and welded knots}

A \textit{virtual knot diagram} is an immersion of circles into the plane where each double point is decorated as either a classical crossing or a virtual crossing. A \textbf{virtual knot} (resp. a \textbf{welded knot}) is an equivalence class of
virtual knot diagrams modulo planar isotopies and \textit{extended Reidemeister moves}, which are depicted in Figure 2 in \cite{satoh2000virtual} where move $D$ is not allowed (resp. $D$ is allowed). A topological interpretation relevant to this paper is presented next.

\subsubsection{Ribbon surfaces}
An orientable knotted surface $S\subset S^4$ is \textit{ribbon} if it bounds an immersed handlebody in $S^4$ with only ribbon intersections. A {\it ribbon intersection} is a disk in $S^4$ which is the image of a pair of disks in the handlebody, one properly embedded in the handlebody (so that its boundary lies on $S$) and one embedded in the interior of the handlebody. A neighborhood of the latter disk inside the handlebody can be pushed into $B^5$ to obtain an embedded handlebody in $B^5$ with boundary $S$ and with only index 0 and 1 critical points.

\subsubsection{The Tube map}
Satoh defined the \textbf{Tube map} in \cite{satoh2000virtual}, which takes a virtual knot (arc) diagram to a diagram of a ribbon torus (sphere), and proved that this respects the virtual Reidemeister moves and the welded move, i.e.\ if two diagrams are welded equivalent, then their tubes are isotopic. Moreover, he proved the Tube map is surjective, i.e.\ every ribbon torus or sphere is the tube of a welded knot or arc. The tube map gives a convenient combinatorial way to encode ribbon surfaces, which is compatible with bridge trisections.

\subsubsection{Bridge numbers of virtual and welded knots}

Virtual and welded bridge numbers have been studied in \cite{boden2015bridge} and the references therein. The definitions for $\beta_O(K),\beta_{\mathbb{R}^2}(K),$ and $\omega(K)$ are identical to the ones in Sections \ref{overpass perspective}, \ref{plane perspective}, and \ref{wirtinger perspective} except we replace every instance of "a classical knot diagram $D$" appearing in those definitions to "a virtual knot diagram $D$". Similarly, if $K$ is a welded knot, then we minimize $\beta_O,\beta_{\mathbb{R}^2}$ and $\omega$ over welded equivalent diagrams representing $K$. Due to Nakanishi and Satoh, $\beta_O(K) \neq \beta_{\mathbb{R}^2}(K)$ for virtual knots, but the two quantities are equivalent for welded knots \cite{nakanishi2015two}.

To demonstrate the definitions on virtual knot diagrams, consider a diagram $D$ representing $K$ in Figure \ref{fig:demonstrate}. Note that $\beta_{\mathbb{R}^2}(D) = 3$ since $D$ has three minima with respect to the standard height function on the page. The overpasses of $D$ are represented by bold segments in the rightmost picture of Figure \ref{fig:demonstrate} so $\beta_O(D) = 2.$ Finally, observe that $\omega(D) = 1$ as $D$ is 1-colorable, with the seed indicated in the leftmost picture of Figure \ref{fig:demonstrate}. This demonstrates the effectiveness of the Wirtinger number. The main result of \cite{pongtanapaisan2019wirtinger} shows that $\omega(K) = \beta_O(K)$ for virtual knots. Using a very similar argument, one can deduce that $\omega(K) = \beta_O(K)$ for welded knots as well.

\begin{figure}
    \centering
    \includegraphics[width=.7\textwidth]{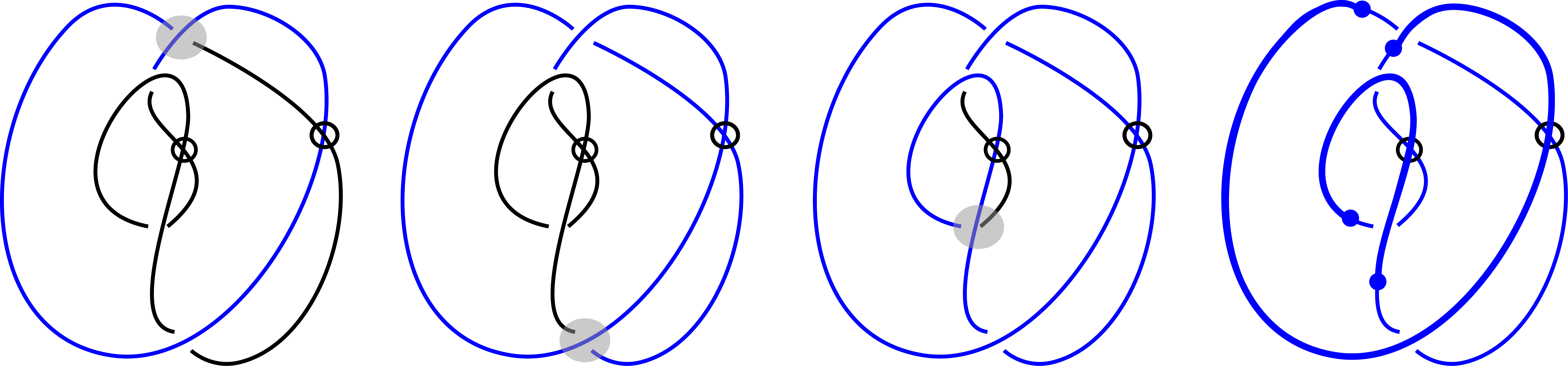}
    \caption{A virtual knot diagram $D$ with $\omega(D) = 1, \beta_O(D) = 2,$ and $\beta_{\mathbb{R}^2}(D) = 3.$ The diagram $D$ represents a nontrivial virtual knot, but a welded trivial knot.}
    \label{fig:demonstrate}
\end{figure}

We are now ready for the construction of ribbon tori whose meridional rank equals the bridge number.

\subsection{BBKM knot diagrams with virtual crossings}

Inspired by ideas from \cite{baader2020coxeter,baader2020bridge}, we will provide families of ribbon $T^2$-knot $F$ such that $\mu(F) = \beta(F).$ We begin by defining some local moves on virtual knot diagrams, which will play an important role in our theorems. We call a replacement of a classical crossing by a virtual crossing a \textit{virtualization}. The \textit{flank-switch move} is shown in Figure \ref{fig:extendflanking}, and the \textit{flank move} is simply the flank-switch move, but we do not switch the classical crossing. Observe that if one performs virtualizations on two adjacent crossings in a twist region, one can remove two newly created virtual crossings by the virtual analog of the Reidemeister II move. Similarly, if one performs two flank-type moves on two adjacent crossings in a twist region, one can use the virtual Reidemeister II move to get rid of two virtual crossings, resulting in two classical crossings in between two virtual crossings. For the bridge number computations, it is also convenient to notice that the Wirtinger number does not increase after the flank-switch and the flank moves are performed.

\begin{lem}
Suppose that $D'$ is obtained by performing some virtualizations, flank moves, and flank-switch moves on a classical knot diagram $D$ , then $\omega(D') \leq \omega(D).$
\end{lem}\label{lem:virtualflankwirtinger}
\begin{proof}
There is a moment during the coloring sequence such that the overstrand at a crossings $c$ of $D$ is colored and the incoming strand is colored. The coloring move can be performed to extend the coloring to the outgoing strand. We break into two cases.

\textbf{Case 1:} If $c$ is replaced by a virtual crossing, one can simply omit this coloring move at $c$ from the coloring sequence on the virtualized diagram. This shows that the seeds for $D$ persists as seeds for $D$ after some classical crossings have been virtualized.

\textbf{Case 2:} If $c$ is replaced by virtual 2-string tangles in the definitions of the flank-switch move and the flank move, then the strands touching the northeast and the northwest endpoints of the virtual tangle are colored blue. The coloring moves can still be performed to extend the coloring to all strands in this virtual 2-string tangle. Figure \ref{fig:extendflanking} illustrates this behavior when one flank-switch move is performed.

In conclusion, $\omega(D') \leq \omega(D)$ as claimed.
\end{proof}

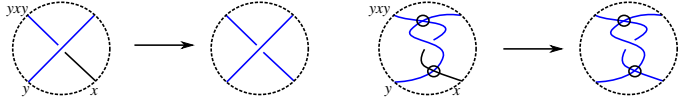
\begin{figure}[!ht]
\labellist
\small\hair 2pt
\centering

\pinlabel \tiny\text{$x$} at 296 15
\pinlabel \tiny\text{$y$} at 51 15
\pinlabel \tiny\text{$yxy$} at 20 304

\pinlabel \tiny\text{$x$} at 1596 15
\pinlabel \tiny\text{$y$} at 1351 15
\pinlabel \tiny\text{$yxy$} at 1320 304

\endlabellist
\includegraphics[width=0.7\textwidth]{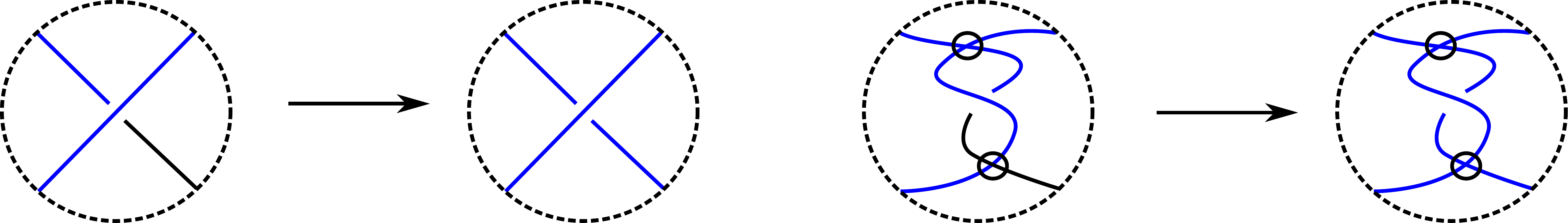}
  \caption{Performing a coloring move after a flank-switch move is performed.}
  \label{fig:extendflanking}
\end{figure}

 Classical BBKM knots are defined in Section \ref{sec:bbkm}. For the remaining of this section, $F$ will be a ribbon $T^2$-knot admitting a virtual knot diagram that is obtained by performing  moves on a classical BBKM knot diagram following the instructions in Theorem \ref{thm:virtualbbkm}. 

\begin{proof}[Proof of Theorem \ref{thm:virtualbbkm}]

(1) For the case of a single Coxeter generator in a twist region, virtualizing any number of them will preserve the labeling since all strands in the twist region are labeled by the same label. Now, we consider the case of two Coxeter generators in a twist region. Suppose that the two bottom endpoints of a vertical twist region is labeled with two distinct meridians $x$ on the left and $y$ on the right. \\
\textbf{Case 1:} Suppose an odd number of virtualizations are performed. At the top of the tangle from left to right is either $(xy)^{2k}x$ and $(xy)^{2k+1}x$ or $(xy)^{2k+1}y$ and $(xy)^{2k}y$. Since we want the bottom label and the top label on the same side to be equal, we have that $(xy)^{2k} = 1$ and $(xy)^{2k+2} = 1$. This means that $(xy)^2 = 1$ and therefore we can let the Coxeter weight corresponding to the virtualized twist region be 2. \\
\textbf{Case 2:} Suppose an even number of virtualizations are performed. At the top of the tangle from left to right is either $(xy)^{2k}x$ and $(xy)^{2k-1}x$ or $(xy)^{2k}y$ and $(xy)^{2k-1}y$. Since we want the bottom label and the top label on the same side to be equal, we have that $(xy)^{2k} = 1$ and we can let the Coxeter weight corresponding to the virtualized twist region be 2$k$.\\

(2) The flank-switch move preserves the original labeling as shown in Figure \ref{fig:extendflanking}.\\

(3) Note that each time one component of the integral tangle goes under the other component, the labeling of the two sides of the bigon after that crossing have exponents that jump from two sides of the bigon before that crossing by a certain interval $r$. After one instance of flanking, the exponents of $(ab)$ will drop by $r$. Adding a new classical crossing will cancel out this drop. Figure \ref{fig:labelrattang} demonstrates a specific example $r=18$ with the jumps in the exponent shown.
\end{proof}

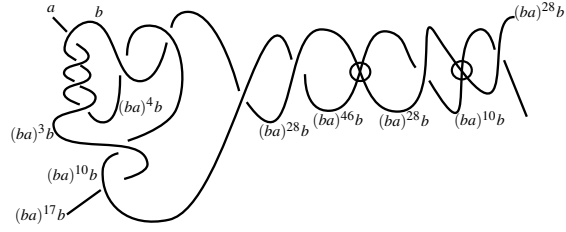
\begin{figure}
\labellist
\small\hair 2pt

\pinlabel \tiny\text{$(ba)^{17}b$} at -44 29

\pinlabel \tiny\text{$(ba)^{10}b$} at 74 179

\pinlabel \tiny\text{$(ba)^3b$} at -74 329
\pinlabel \tiny\text{$(ba)^{4}b$} at 326 429
\pinlabel \tiny\text{$(ba)^{28}b$} at 856 339
\pinlabel \tiny\text{$(ba)^{46}b$} at 1056 380
\pinlabel \tiny\text{$(ba)^{28}b$} at 1299 377
\pinlabel \tiny\text{$(ba)^{10}b$} at 1579 380
\pinlabel \tiny\text{$a$} at -4 790
\pinlabel \tiny\text{$b$} at 174 780
\pinlabel \tiny\text{$(ba)^{28}b$} at 1799 780

\endlabellist
 \centering
    \includegraphics[width=.5\textwidth]{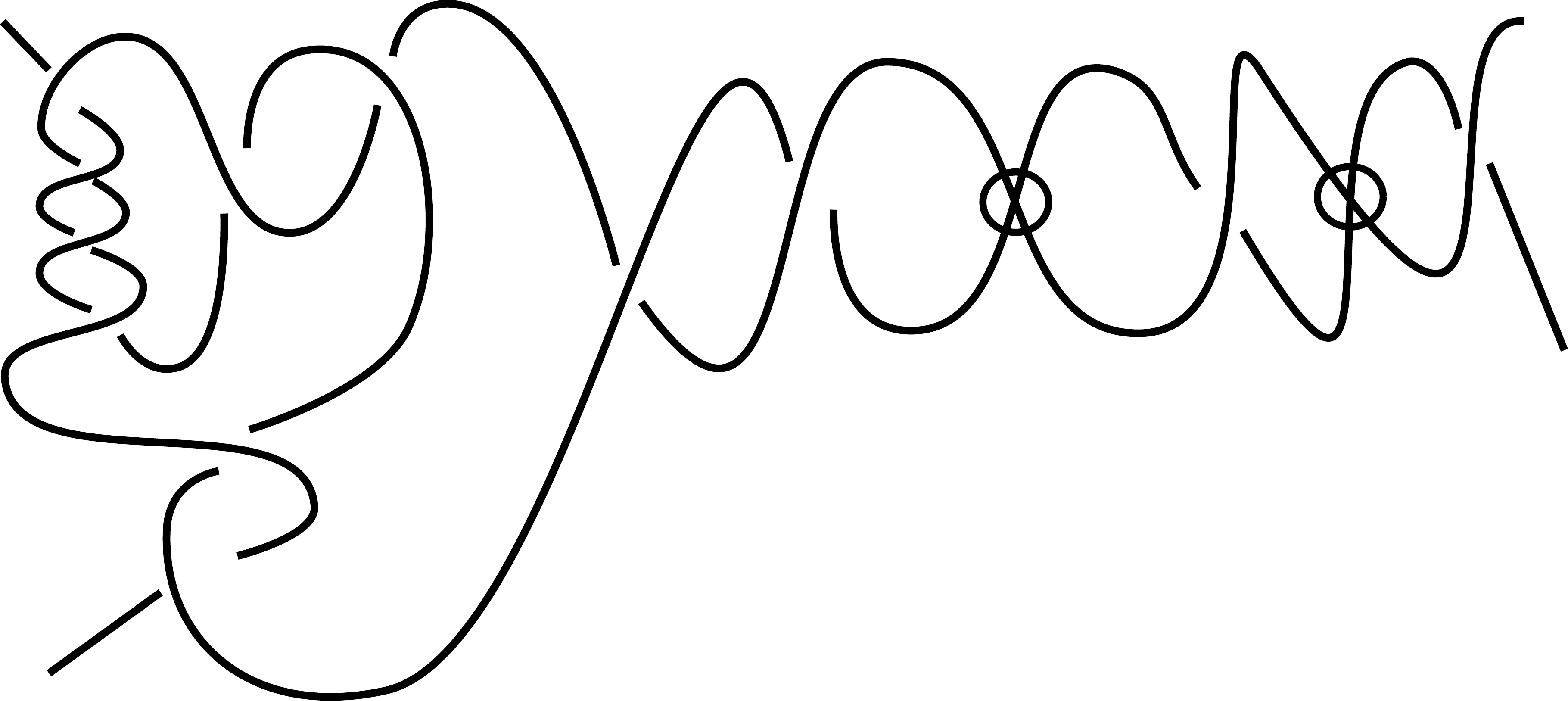}
    \caption{Flanking and balancing out in a rational tangle.}
    \label{fig:labelrattang}
\end{figure}

\VBBKM*

\begin{proof}
    Starting with a classical BBKM knot $K$ with $\mu(K)=n$ represented as a diagram $D$. Theorem \ref{thm:virtualbbkm} shows that performing the flank-switch moves, flank moves, and virtualization moves as instructed gives a virtual knot diagram $D'$ such that $\mu(Tube(D'))\geq n$.

    For the matching upper bound, Lemma \ref{lem:virtualflankwirtinger} says that the same seeds that give rise to $\omega(D)=n$ persist to be seeds for the virtual knot diagram $D'.$ By Theorem 3.3 of \cite{pongtanapaisan2019wirtinger}, $D'$ admits a diagram with $n$ overpasses. By Theorem 1.2 of \cite{nakanishi2015two}, there are local moves that one can perform to transform $D'$ to a diagram $D''$ with $n$ minima with respect to the standard height function on $\mathbb{R}^2$. Furthermore, one can require that Tube($D''$) is isotopic to Tube$(D')$ after the moves. These $n$ minima gives rise to $n$ 0-handles for the ribbon surface Tube$(D') = F$. Thus, $\beta(F) \leq n$. To see an illustration of the correspondence, the readers can consult Figure 9 of \cite{JMMZ}, where each maximum corresponds to an unknot component in a band diagram presentation of a knotted surface. Since each 0-handle corresponds to an unknot component and each 1-handle corresponds to a band, the corollary is proved.
\end{proof}

\section{Applications to Bridge Trisections}\label{sec:bridgetrisections}

In this section we collect the applications of our results to bridge trisections. Indeed, the use of meridional rank in \cite{meier2017bridge} was one of the primary inspirations for this work.

\subsection{Bridge trisections of knotted surfaces in $S^4$}

 In \cite{meier2017bridge}, Meier and Zupan showed that any knotted surface can be decomposed into a union of three \textit{trivial disk systems}, called a \textit{bridge trisection}. Bridge trisections have 4 parameters $(b;c_1,c_2,c_3)$, satisfying $\chi(S)=c_1+c_2+c_3-b$. In this section we point out that $\beta(S)=\min\{c_i\}$, hence bridge number provides the lower bound $3\beta(S)-\chi(S)\leq b(S)$ for bridge trisection index.

A $c$-component $D^2$-tangle is often simply referred to as a trivial \textit{$c$-disk system}. A \textit{$(b;c_1,c_2,c_3)$-bridge trisection} of a knotted surface $S \subset S^4$ is a decomposition $(S^4,S) = (X_1,\mathcal{D}_1) \cup (X_2,\mathcal{D}_2) \cup (X_3,\mathcal{D}_3)$ satisfying the following properties:
\begin{enumerate}
    \item $(X_i,\mathcal{D}_i)$ is a trivial $c_i$-disk system.
    \item $(\beta_{ij},\alpha_{ij})$ = $(X_i,\mathcal{D}_i) \cap (X_j,\mathcal{D}_j)$ is a trivial $b$-strand tangle.
    \item $(\Sigma,\mathbf{p}) = (X_1,\mathcal{D}_1) \cap (X_2,\mathcal{D}_2) \cap (X_3,\mathcal{D}_3)$ is a 2-sphere, and $\mathbf{p}$ is a collection of $2b$ points contained in $\Sigma.$
\end{enumerate}

The \textbf{bridge trisection index} $b(S)$ of a knotted surface $S$ in $S^4$ is defined to be the minimum number $b$ such that $S$ admits a $(b;c_1,c_2,c_3)$-bridge trisection. The quantity $\min\{c_i\}$ is called the \textbf{patch number}, because it is the minimal number of disks, or patches, in one of the disk systems.



\begin{prop}
Let $S$ be a knotted surface. Then $\beta(S)=\min\{c_1,c_2,c_3\}$, where the minimum is taken over all $(b;c_1,c_2,c_3)$-bridge trisections of $S$.
\end{prop}

\begin{proof}
If $S$ admits a $(b;c_1,c_2,c_3)$-bridge trisection, then Meier-Zupan show in \cite{meier2017bridge} that $S$ has an embedding with $c_i$ minima, $b-c_j$ saddles, and $c_k$ maxima, for any choice of distinct $i,j,k$. Thus $S$ has an embedding with $\min\{c_i\}$ minima, and $\beta(S)\leq \min\{c_i\}$.

For the reverse direction, start with an embedding of $S$ in Morse position with $\beta=\beta(S)$ minima. Note that $S$ can be isotoped so that its critical points are index-ordered, without changing how many critical points of each index exist. Slicing $S^4$ just above the minima is then an unlink $L$ with $\beta$ components, and the index 1 critical points give rise to a set of bands $b$ such that surgering $L$ along $b$ results in the unlink just below all the maxima. See for instance \cite{kawauchi2012survey}. As in \cite{meier2017bridge}, this banded unlink $(L,b)$ can then be isotoped into a banded bridge splitting, again without changing the number of components. The banded bridge splitting induces a $(b;c_1,c_2,c_3)$ bridge trisection, where $c_1=\beta$. Thus $\beta\geq \min\{c_i\}$.

\end{proof}

\begin{cor}
Let $S$ be a knotted surface. Then $3\beta(S)-\chi(S)\leq b(S)$.
\end{cor}

\begin{proof}
Combine $\beta(S)\leq\min\{c_i\}$ and the Euler characteristic calculation $\chi(S)=c_1+c_2+c_3-b$.
\end{proof}

As expected, in all the cases that we can prove $\mu(S)=\beta(S)$, we can construct a balanced $(3\beta-\chi(S);\beta)$-bridge trisection.

\begin{quest}
Does there exist a surface $S$ with $b(S)>3\beta(S)-\chi(S)$?
\end{quest}

In the case of twist-spun 2-knots, this is Question 5.2 of \cite{meier2017bridge}, which we partially answer with Theorem \ref{thm:twistspunrank}. 

\begin{cor}[of Theorem \ref{thm:twistspunrank}]
Let $S$ be a surface $\tau^n K \#^m\R P^2$, satisfying the hypotheses of Theorem \ref{thm:twistspunrank}. Then $b(S)=3\beta(S)-\chi(S)$.
\end{cor}

These examples reprove the following theorem, due partially to \cite{meier2017bridge} and also \cite{sato2020bridge}.

\begin{thm}[\cite{meier2017bridge},\cite{sato2020bridge}]
For any integer $b\geq 4$, there exist infinitely many surfaces with bridge trisection index equal to $b$.
\end{thm}

\begin{proof}
The cases $b\equiv$ 0 or 1 mod 3 are covered by the techniques of \cite{meier2017bridge}: it is pointed out in that paper that if $K$ is a knot in $S^3$ with bridge number $\beta=\mu(K)$, then $\tau^0 K$ admits a $(3\beta-2 ;\beta)$-bridge trisection. Since the meridional rank of a spun knot $\tau^0 K$ is equal to the meridional rank of $K$, and since the bridge trisections constructed in \cite{meier2017bridge} have parameters $(3\beta-2;\beta)$, we get that $b(\tau^ 0 K)=3\beta-2$ whenever $\mu(K)=\beta(K)$. For the 0 mod 3 case, simply spin the entire circle to obtain the spun-torus of $K$, which still has the same group and meridians of $K$, so has the same meridional rank (and the corresponding bridge trisection has parameters $(3\beta ;\beta)$). Alternatively, one could form the connected sum $\tau^0 K \# T$, where $T$ is a trivial torus, since this does not change the group. The 2 mod 3 case can still be handled by orientable surfaces when $\chi$ is even, e.g.\ $b(\tau^0 K\#^2 T)=8$, but when $b\equiv$ 5 mod 6 we must connect sum with a standard $\R P^2$ instead, as in \cite{sato2020bridge}. In this case we can apply Theorem \ref{thm:twistspunrank} to $\tau^{2m} K\# \R P^2$ for any BBKM knot $K$ to prove the theorem.
\end{proof}

We remark that Theorem \ref{thm:twistspunrank}.2 provides many examples of odd-twist spun knots for which we establish $\beta$ and $b$. The case of odd-twist spun knots was left open by \cite{sato2020bridge}, and is not attainable by Coxeter quotients or by considering Fox colorings.

\subsection{The Tube map and bridge trisections}

An advantage to the computations of the bridge number and meridional rank of virtual knot diagrams is that we can automatically get the bridge trisection index of its Tube, using the construction of \textit{ribbon bridge trisections}, formulated in \cite{JMMZ}. In that work, ribbon bridge trisections were defined and utilized by the first author and Meier, Miller, and Zupan to provide the first examples of nonisotopic bridge trisections of isotopic knotted surfaces \cite{JMMZ}. The following proposition gives a way to compute an upper bound for the bridge trisection number of a ribbon $T^2$-knot.

\begin{prop}[\cite{JMMZ}]
Let $S$ be a ribbon $T^2$-knot. Suppose that $S$ can be represented by a virtual knot diagram $D$ such that $\beta_{\mathbb{R}^2}(D) = r,$ then $b(S) \leq 3r$.\label{weldedtotrisection}
\end{prop}
\begin{proof}
The equatorial cross-section $K$ of $Tube(D)$ has a diagram with $2r$ minima with respect to $h$ as shown on the right of Figure \ref{fig:virtualbraid} (without the colored bands). The link $K$ together with these bands gives a movie description of ribbon annuli, whose double is $Tube(D)$. A banded unlink diagram of $Tube(D)$ can then be obtained by performing a band move and attaching a dual band (shown in fuchsia) near each yellow band as depicted in the middle of Figure \ref{fig:perturbing}. To obtain a dual bridge disk for each yellow band, we can perturb the banded unlink as shown on the right of Figure \ref{fig:perturbing}. The process described above produces a banded bridge splitting, and thus a bridge trisection, with $3r$ bridges.
\end{proof}

The above proposition, together with Theorem \ref{thm:virtualbbkm}, yield the following corollary.


\bridgenumberribbon*

\begin{figure}[ht!]

    \centering
    \includegraphics[width=.4\textwidth]{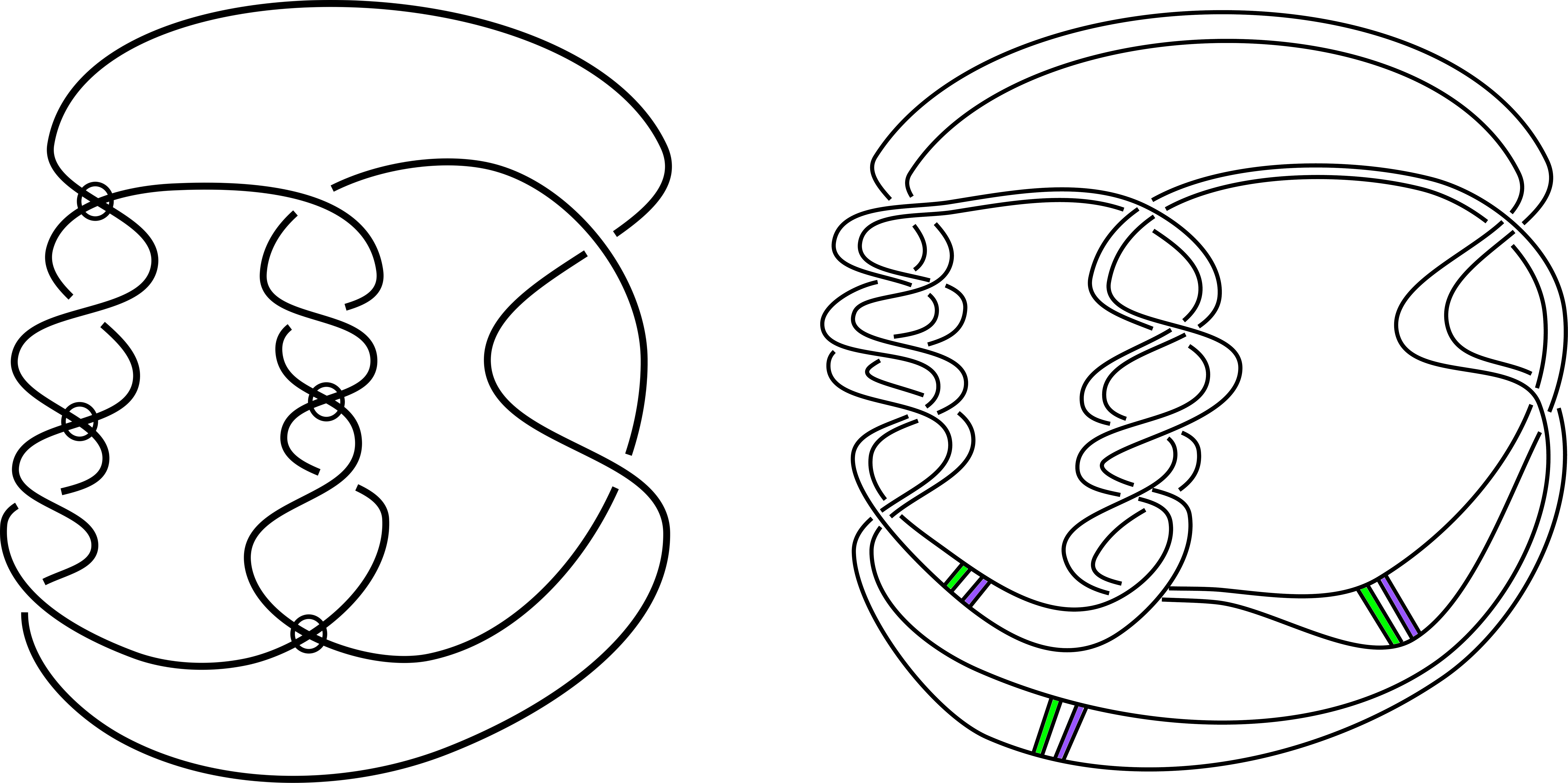}
    \caption{From a virtual knot diagram $D$ with $r$ minima (left), one can get a banded link diagram for $Tube(D)$ with $2r$ minima (right).}
    \label{fig:virtualbraid}
\end{figure}
\begin{figure}[ht!]
    \centering
    \includegraphics[width=.7\textwidth]{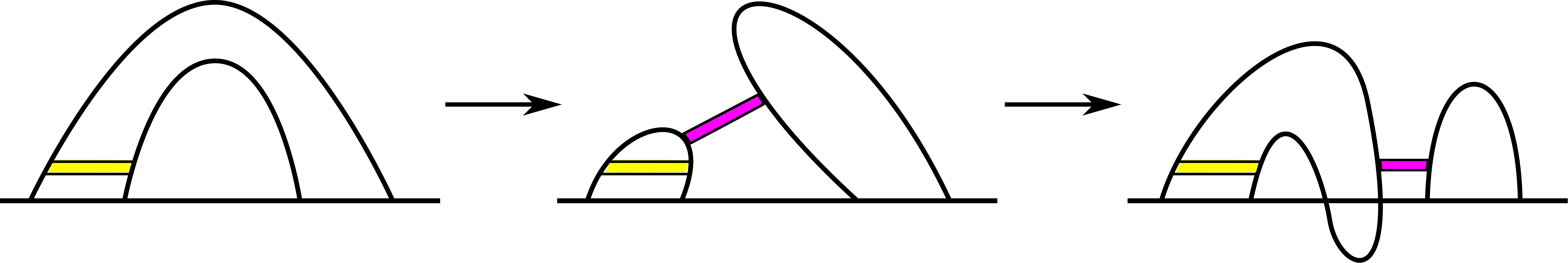}
    \caption{To get a bridge trisection, we perform a perturbation to get the rightmost image. This increases the bridge index by 1 for each of the original minima of $D$, resulting in a bridge trisection index of $3r$.}
    \label{fig:perturbing}
\end{figure}


\bibliographystyle{plainurl}
\bibliography{ref}

\end{document}